\documentclass[ps,12pt,twoside,a4paper]{article}
\usepackage{braket, amsfonts}
\usepackage{amsmath}
\usepackage{subcaption}
\usepackage{amsopn}
\usepackage{amsfonts}
\usepackage{amsthm}
\usepackage{graphicx}
\usepackage{algpseudocode}
\usepackage[section]{algorithm}
\usepackage{dsfont}
\usepackage{hyperref}

\title{Projected Stochastic Gradients for Convex Constrained Problems in Hilbert Spaces}
\author{Caroline Geiersbach\thanks{University of Vienna, Austria (\texttt{caroline.geiersbach@univie.ac.at})} \and Georg Pflug\thanks{University of Vienna, Austria (\texttt{georg.pflug@univie.ac.at})}}
\newcommand{\R}{\mathbb{R}}

\newcommand{\N}{\mathbb{N}}
\newcommand{\E}{\mathbb{E}}
\newcommand{\pP}{\mathbb{P}}
\newcommand{\D}{\text{ d}}

\newcommand{\U}{\mathcal{U}_{ad}}
\newcommand{\tu}{\tilde{u}}

\DeclareMathOperator*{\argmin}{arg\,min}
\DeclareMathOperator*{\esssup}{ess\,sup}
\numberwithin{equation}{section}

\newtheorem{thm}{Theorem}[section]
\newtheorem{lemma}[thm]{Lemma}

\newtheorem{prop}[thm]{Proposition}
\newtheorem{assumption}[thm]{Assumption}
\newtheorem{cor}[thm]{Corollary}
\theoremstyle{definition}
\newtheorem{rem}[thm]{Remark}

\begin{document}
\maketitle
\begin{abstract}
Convergence of a projected stochastic gradient algorithm is demonstrated for convex objective functionals with convex constraint sets in Hilbert spaces. In the convex case, the sequence of iterates ${u_n}$ converges weakly to a point in the set of minimizers with probability one. In the strongly convex case, the sequence converges strongly to the unique optimum with probability one. An application to a class of PDE constrained problems with a convex objective, convex constraint and random elliptic PDE constraints is shown. Theoretical results are demonstrated numerically.
\end{abstract}

\section{Introduction}
We consider problems of the form
\begin{equation}\label{eq:SAproblem}
 \min_{u \in C} \{j(u) = \E[J(u,\xi)] \},
\end{equation}
where $C$ is a nonempty, closed and convex subset of a Hilbert space $H$. The random vector $\xi:\Omega \rightarrow \R^m$ is defined on a probability space $(\Omega, \mathcal{F}, \pP)$; it is assumed that for every $\omega$, $u \mapsto J(u,\xi(\omega))$ is convex on $C$, making $j$ convex as well. Additionally, we require that $J$ is $L^2$-Fr\'echet differentiable\footnote{$J(u,\xi(\omega)):H \times \Omega \rightarrow \R$ is $L^2$-Fr\'echet differentiable if for an open $U \subset H$ there exists a bounded and linear random operator $A:U \times \Omega \rightarrow \R$ such that $\lim_{h \rightarrow 0} \lVert J(u + h, \xi) - J(u, \xi) + A(u,\xi)h \rVert_{L^2(\Omega)} / \lVert h \rVert_H = 0$, where $L^2(\Omega)$ is the space of square integrable functions on $(\Omega, \mathcal{F}, \pP).$}
with respect to $u$ on an open neighborhood of $C$, which implies that $j:H \rightarrow \R$ is Fr\'echet differentiable. We assume
\begin{equation}\label{eq:expectation}
\E[J(u,\xi)] = \int_\Omega J(u,\xi(\omega)) \D\pP (\omega)
\end{equation}
is well-defined and finite for each $u \in C.$ Unless the support of $\pP$ is finite and small, the direct calculation of the integral \eqref{eq:expectation} is not tractable. A common approximation method for this integral involves sampling, where it is assumed that it is possible to generate a random independent identically distributed (i.i.d.)~sample $\xi_1, \dots, \xi_N$  with $\xi_i:=\xi(\omega_i)$ and $\omega_i \in \Omega$.  In a \textit{sample average approximation} (SAA) approach, the problem \eqref{eq:SAproblem} is replaced by an approximate problem
$$\min_{u \in C} \{\hat{j}_N(u) = \frac{1}{N} \sum_{i=1}^N J(u,\xi_i) \},$$
which is then solved as a proxy for the basic problem \eqref{eq:SAproblem}. Notice that in the SAA approach the number of samples is fixed a priori and the approximate problem does not contain any randomness so can be solved by any appropriate optimization software. For an overview on the SAA methods see the chapter ``Monte Carlo Sampling Methods'' in Shapiro \cite{Shapiro2003}.

In contrast, the {\em stochastic (quasi)-gradient} or {\em stochastic approximation} technique does not require the sample size to be determined a priori. The iterative optimization step  relies on the notion of a stochastic  gradient, i.e.~a random function $G(u,\xi)$ such that $\E[G(u,\xi)] \approx \nabla j(u)$.
A stopping criterion and the determination of confidence regions can be based on information gathered during the iteration, which gives an advantage over a-priori rules.
The stochastic approximation technique originated in a paper by Robbins and Monro in 1951 \cite{Robbins1951}, where the authors developed an iterative method for finding the root of a function where only noisy estimates of the function are available. A related work for finding the maximum of a regression function followed in a paper by Kiefer and Wolfowitz in 1952 \cite{Kiefer1952}.

In PDE constrained optimization, the use of stochastic approximation techniques is still unexploited.  It is the goal of this paper to establish convergence for convex problems in Hilbert spaces, and demonstrate its application on a particular class of problems, namely a convex problem with random elliptic PDE constraints and additional convex constraints. Variants of the model problem in this paper have been investigated in various works; approaches typically hinge on a finite-dimensional noise assumption introduced in \cite{Babuska2004}, which allows for a parametric representation of the random fields. Kouri et al.\,\cite{Kouri2014} used a parametric representation of random fields, as well as a trust-region algorithm with sparse grids. Hou, Lee and Manouzi \cite{Hou2011} relied on a Wiener-It\^o chaos expansion combined with a finite element approximation to deduce a deterministic system. Rosseel and Wells \cite{Rosseel2012} considered the problem where the control is also uncertain, and developed a one-shot approach, using a series expansion for the random field and comparing stochastic collocation to the stochastic Galerkin method.

The paper is structured as follows. In section~\ref{sec:PSG}, the projected stochastic gradient algorithm is defined, notation is introduced and existing convergence results are summarized. In section~\ref{sec:convergence}, convergence is proven. A model problem is introduced in section~\ref{sec:ModelProblem} and it is shown that the problem satisfies the conditions for convergence. In section~\ref{sec:Numerics}, the algorithm is demonstrated using numerical experiments. Closing remarks are prepared in section~\ref{sec:Conclusion}.

\section{Algorithm, Notation, and Existing Results}\label{sec:PSG}
We denote the inner product in $H$ as $\langle \cdot, \cdot \rangle$ and induced norm $\lVert \cdot \rVert = \sqrt{\langle \cdot, \cdot \rangle}$. We introduce the notation $u_n \rightarrow u$ for the strong convergence and $u_n \rightharpoonup u$ for weak convergence in $H$. The projection onto a closed convex set $C$ is denoted by $\pi_C:H \rightarrow C$ and is defined as the function such that
$$ \pi_C(u) = \underset{w \in C}{\argmin} \,  \lVert u-w \rVert.$$
The projected stochastic gradient (PSG) algorithm, which is studied in this paper, is summarized in Algorithm~\ref{alg:PSGD_Hilbert}.
\begin{algorithm}
\begin{algorithmic}[1] 
\State \textbf{Initialization:} $u_1 \in H$
\For{$n=1,2,\dots$}
\State Generate $\xi_n$, independent from $\xi_1, \dots, \xi_{n-1}$, and $\tau_n >0$
\State $u_{n+1} := \pi_{C}(u_n - {\tau_n}G(u_n,\xi_n))$
\EndFor
\end{algorithmic}
\caption{Projected Stochastic Gradient (PSG)}
\label{alg:PSGD_Hilbert}
 \end{algorithm}

A natural choice for a stochastic gradient is $G(u,\xi) = \nabla_u J(u,\xi)$, but the gradient can even be chosen to allow for some bias; see section~\ref{sec:convergence}. Iterates $u_{n}$ are a function of the history $(\xi_1, \dots, \xi_{n-1})$ and are therefore random. The strength of this method is the low memory requirement -- only the current iterate $u_n$ needs to be stored to compute the next step -- and its easy adaptability to deterministic gradient-based solvers. Its performance is however quite sensitive to a proper choice of step size, and the projection onto $C$ might be as complex as the original problem \eqref{eq:SAproblem}. In the deterministic case, it is possible take larger steps by using step sizes that guarantee descent, i.e.~ensuring $j(\pi_C[u_n - \tau_n \nabla j(u_n)]) \leq j(u_n)$; see for instance \cite{Hinze2009} for a projected Armijo rule. In the stochastic case, exogenous step size rules of the form
 \begin{equation}\label{eq:RobbinsMonroStepSize}
\tau_n \geq 0, \quad \sum_{n=1}^\infty \tau_n = \infty, \quad  \sum_{n=1}^\infty \tau_n^2 < \infty
 \end{equation}
are a common requirement to ensure convergence. For this reason, Algorithm \ref{alg:PSGD_Hilbert} is not a descent method. To terminate the algorithm, one relies on efficiency estimates, which are summarized in subsection~\ref{subsec:stepsizerules}.

Convergence of the stochastic gradient algorithm is well-established in finite-dimensional spaces. For unconstrainted problems (i.e.~$C = \R^d$), Bertsekas and Tsitsiklis \cite{Bertsekas2000} proved that, with a Lipschitz continuous gradient and step sizes diminishing to zero, $\lim_{n \rightarrow \infty}j(u_n) = -\infty$ or $j(u_n)$ converges to a finite value and $\lim_{n \rightarrow \infty} \nabla j(u_n) =0$ a.s. Convergence of the projected stochastic gradient method in the presence of zero-mean noise and systematic error was handled e.g.~by Pflug \cite{Pflug1996}, but also by many other authors; see for instance the work by Kushner and Yin \cite{Kushner2003}.

In Hilbert spaces, most results for constrained convex optimization are in the deterministic or nonsmooth setting. Poljak \cite{Poljak1967} proved that the sequence $\{u_{n} \}$ contains a minimizing subsequence $\{u_{n_k} \}$ such that $j(u_{n_k}) = \inf_{u \in C} j(u)$ with iterations of the form $u_{n+1}=\pi_C(u_n+v_n),$ where $v_n$ is a support functional of $j$ and subject to the rule $\lim_{n \rightarrow \infty} \lVert v_n \rVert= 0$ and $\sum_{n=0}^\infty \lVert v_n \rVert = \infty$. For constrained nonsmooth convex optimization, Alber, Iusem, and Solodov \cite{Alber1998} studied nonsmooth convex optimization and proved weak convergence of the generated sequence to a minimizer if the problem has solutions, and unboundedness of the sequence otherwise. Bello Cruz and de Oliveira \cite{Cruz2016} established weak convergence of the generated sequence to a minimizer in the case of a convex, G\^{a}teaux differentiable objective function, and presented a modified projected gradient method where strong convergence of the sequence can be proven.

Some papers have treated infinite dimensional stochastic approximation; of note are \cite{Venter1966}, \cite{Nixdorf1984}, \cite{Chen2002}.
Goldstein \cite{Goldstein1988} proved almost sure convergence to the minimum in the unconstrained case where $j$ achieves a unique mimimum. Yin and Zhu \cite{Yin1990} studied processes of the form $u_{n+1}=u_n + \tau_n(f(u_n) + w_n)$ for correlated noise $\{w_n \}$ sequences and nonlinear operators $f$. In particular, almost sure convergence was demonstrated even if $u$ does not satisfy the linear growth condition $|f(u)| \leq K(1+ \lVert u\rVert)$. Barty, Roy and Strugarek \cite{Barty2005} proved strong convergence of $u_{n+1} = \pi_C(u_n + \tau_n(v_n + w_n))$ in the case where $j$ is strongly convex, and proved that in the general convex case, $j(u_n) \rightarrow j(\bar{u})$ for an accumulation point $\bar{u}$ of the sequence. An anonymous reviewer brought to our attention the paper by Culioli and Cohen \cite{Culioli1990}, where a convergence result similar to ours was produced. Our result, however, includes a bias term and also shows convergence of the sequence $\{ u_n\}$ to a specific random point in the solution set, thus precluding the possibility of oscillations within the set of solutions.

We emphasize the following features of our analysis that makes it different from existing above results:
\begin{itemize}
 \item Almost sure weak convergence of the entire sequence $\{ u_n\}$ to a specific point in the solution set is established as long as a solution exists.
 \item All that is needed to establish convergence is convexity of $j$ and at most quadratic growth for the second moment of the stochastic gradient over the constraint set $C$. No assumption about Lipschitz continuity of the gradient is needed.
 \item Efficiency estimates are derived for the case of an unbounded constraint set $C$.
\end{itemize}
To our knowledge, the application to PDE constrained optimization under uncertainty is also novel.

\section{Convergence Result}
\label{sec:convergence}
In this section, we prove convergence of Algorithm~\ref{alg:PSGD_Hilbert} for general convex problems in Hilbert spaces. The proof relies on the use of martingale methods as in \cite{Pflug1996} for the finite-dimensional case. We recall that given a probability space $(\Omega, \mathcal{F}, \pP)$, a (discrete) filtration $\{ \mathcal{F}_n\} \subset \mathcal{F}$ is an increasing sequence of $\sigma$-algebras. A stochastic process $\{\beta_n\}$ is said to be adapted to a filtration $\{ \mathcal{F}_n \}$ iff $\beta_n$ is $\mathcal{F}_n$-measurable for all $n \in \N$. The natural filtration is the filtration generated by the sequence $\beta_n$ itself and is given by $\mathcal{F}_n = \sigma(\beta_m: m \leq n)$. If for an event $F \in \mathcal{F}$, it holds that $\pP[F] = 1$, we say $F$ occurs almost surely and denote this with a.s.

To proceed, we recall some technical results.

\begin{lemma}[Robbins-Siegmund] \label{lemma:Robbins-Siegmund}
Let $\{\mathcal{F}_n\}$ be an increasing sequence of $\sigma$-algebras and 
$v_n$, $a_n$, $b_n$, $c_n$ nonnegative random variables adapted to $\mathcal{F}_n.$ If
\begin{equation}\label{eq:Robbins-Siegmund}
\E[v_{n+1} | \mathcal{F}_n] \leq v_n(1+a_n)+ b_n-c_n,
\end{equation}
and $\sum_{n=1}^\infty a_n < \infty,  \sum_{n=1}^\infty b_n < \infty$ a.s., then with probability one, $\{v_n\}$ is convergent and it holds that $\sum_{n=1}^\infty c_n < \infty$. 
\end{lemma}
\begin{proof}
\cite{Pflug1996}, Appendix L.
\end{proof}

The following proposition is a generalization of Proposition 2 in \cite{Alber1998}.
\begin{prop}\label{prop:technicalproposition}
Let $\{\tau_n \}$ be a nonnegative deterministic sequence and $\{ \beta_n\}$ a nonnegative random sequence in $\mathbb{R}$ adapted to $\{\mathcal{F}_n$\}. Assume that
$\sum_{n=1}^\infty \tau_n = \infty$ and  $\mathbb{E} [\sum_{n=1}^\infty \tau_n \beta_n] < \infty$. Moreover assume that 
$\beta_{n} - \mathbb{E}[\beta_{n+1}|\mathcal{F}_n] \leq \gamma \tau_n$ a.s.~for all $n$ and some $\gamma > 0$. Then
$$\beta_n  \hbox{ converges to } 0  \hbox{ a.s. }$$
\end{prop}
\begin{proof}
The assumptions imply that $\liminf_n \beta_n = 0$ a.s. Indeed, if this were not the case, with positive probability, for some $\epsilon>0$ there would exist $N_\epsilon$ such that $\beta_n \geq \epsilon/2$ for all $n\geq N_\epsilon$ and thus $\sum_{n=N_\epsilon}^\infty \beta_n \tau_n \geq \epsilon/2 \sum_{n=N_\epsilon}^\infty \tau_n = \infty$, a contradiction to $\E[\sum_{n=1}^\infty \tau_n \beta_n] < \infty$. One has to also show that $\limsup_n \beta_n = 0$ a.s., which we likewise argue by contradiction. Suppose that there exists an $\epsilon > 0$ such that 
\begin{equation}
\label{eq:lim-sup-probability}
\pP\{ \limsup_n \beta_n > 2 \epsilon \} = \eta > 0.
\end{equation}
Define the following stopping times
\begin{eqnarray*}
m_1 &=& \inf \{n : \beta_n > 2\epsilon \}\\
\ell_k &=& \inf \{n> m_k : \beta_n < \epsilon \}; k\ge 1.\\
m_{k+1} &=& \inf \{n>\ell_k : \beta_n > 2 \epsilon \}
\end{eqnarray*}
The assumption \eqref{eq:lim-sup-probability} implies that with probability at least $\eta$, infinitely many stopping times are smaller than infinity, since we already established $\pP\{\liminf_n \beta_n = 0\} =1$. On the set where $m_k=\infty$, $\sum_{i=m_k}^{\ell_k-1} \tau_i = 0$ and $\mathds{1}_{i \geq m_k} = 0$ for all $i$. Notice that if $m_k \leq i \leq \ell_k - 1$, it holds that $\beta_i \geq \epsilon$, so
\begin{equation*}
\mathbb{E}\big[ \sum_{k=1}^\infty \sum_{i=m_k}^{\ell_k-1} \tau_i \big] \le  \mathbb{E} \big[ \sum_{k=1}^\infty \sum_{i=m_k}^{\ell_k-1} \tau_i \beta_i/\epsilon \big] \le \frac{1}{\epsilon}
 \sum_{k=1}^\infty \mathbb{E} \big[ \sum_{i=m_k}^{\ell_k-1} \tau_i \beta_i \big] < \infty.
 \end{equation*}
so in particular
\begin{equation}
\label{eq:boundedness-steps}
\lim_{k \rightarrow \infty} \mathbb{E}\big[ \sum_{i=m_k}^{\ell_k-1} \tau_i \big] = 0.
\end{equation}
Now
\begin{eqnarray*}
\epsilon \eta &\le& \mathbb{E} [\beta_{m_k}- \beta_{\ell_k} ] = 
\mathbb{E} \big[\sum_{i=m_k}^{\ell_k-1} \beta_i - \beta_{i+1} \big] = \mathbb{E}\big[\sum_{i=1}^\infty \mathds{1}_{m_k \leq i < \ell_k} \cdot (\beta_i - \beta_{i+1}) \big]\\
&=& \mathbb{E}\big[\sum_{i=1}^\infty  \mathbb{E}[\mathds{1}_{m_k \leq i < \ell_k} \cdot (\beta_i - \beta_{i+1}) |\mathcal{F}_i] \big] = \mathbb{E}\big[\sum_{i=1}^\infty  \mathds{1}_{m_k \leq i < \ell_k} \cdot (\beta_i - \mathbb{E}[\beta_{i+1} |\mathcal{F}_i]) \big]\\
&\le& \mathbb{E}\big[\sum_{i=1}^\infty  \mathds{1}_{m_k \leq i < \ell_k} \cdot \gamma \, \tau_i  \big] = \gamma \, \mathbb{E} \big[ \sum_{i=m_k}^{\ell_k-1} \tau_i \big],
\end{eqnarray*}
where the third equality follows by the law of total expectation; the fourth equality follows since $\mathds{1}_{m_k \leq i < \ell_k}$ and $\beta_i$ are $\mathcal{F}_i$-measurable; and the second inequality follows from the assumption $\beta_{n} - \mathbb{E}[\beta_{n+1}|\mathcal{F}_n] \leq \gamma \tau_n$. Choosing $k$ so large that  \mbox{$\mathbb{E}[\sum_{i=m_k}^{\ell_k-1} \tau_i] < \epsilon \eta /\gamma$}, which is possible by \eqref{eq:boundedness-steps}, one gets a contradiction.
\end{proof}

For convergence of the projected stochastic gradient method, we need the following assumptions on the objective function and the stochastic gradients.

\begin{assumption}\label{assumption:j}
The functions $u \mapsto J(u,\xi(\omega))$ are convex for almost all $\omega \in \Omega$ and $J$ is $L^2$-Fr\'echet differentiable for all $u$ in an open neighborhood of $C$.
\end{assumption}

\begin{assumption}\label{assumption:gradient}
Let $\{\mathcal{F}_n \}$ be an increasing sequence of $\sigma$-algebras and the  sequence of stochastic gradients generated by Algorithm~\ref{alg:PSGD_Hilbert} be given by $\{G(u_n,\xi_n)\}$.  For each $n$, there exist $r_n$, $w_n$ with
$$r_n = \E[G(u_n,\xi_n) | \mathcal{F}_n] - \nabla j(u_n), \quad w_n = G(u_n, \xi_n) - \E[G(u_n, \xi_n) | \mathcal{F}_n],$$ 
which satisfy the following assumptions: (i) $ u_n$ and $r_n $ are $\mathcal{F}_n$-measurable; (ii) for $K_n:=\esssup_{\omega \in \Omega} \lVert r_n(\omega)\rVert$ it holds that \mbox{$ \sum_{n=1}^\infty \tau_n K_n< \infty$} and \mbox{$\sup_n K_n < \infty$}; (iii) there exist $M_1, M_2 > 0$ such that $\E[\lVert G(u,\xi)\rVert^2 ] \leq M_1 + M_2 \lVert u \rVert^2$ for all $u \in C.$
\end{assumption}

\begin{rem}\label{rem:assumptionsinlemma}
 Assumption~\ref{assumption:gradient} requires that the stochastic gradient has the additive representation $G(u_n,\xi_n) = \nabla j(u_n) + w_n + r_n$. This allows for systematic error in the form of $r_n$, which must decay to zero at the given rate. Bias might be for example in the form of numerical error but also due to the approximation of the gradient. The sequence $\{w_n\}$ represents zero-mean random error and satisfies $\E[w_n | \mathcal{F}_n] = 0$ by definition. Requiring adaptivity of the sequences $\{ u_n\}$ and $\{r_n\}$ is automatically satisfied if $\{\mathcal{F}_n\}$ is chosen to be the natural filtration. The final assumption is a growth condition for the second moment of the stochastic gradient over the constraint set $C$.
\end{rem}

\begin{thm}\label{lemma:convergenceofPSG}
Suppose that Assumption~\ref{assumption:j} and Assumption~\ref{assumption:gradient} hold. If there exists a $\tu \in C$ such that $j(\tu) \leq j(u)$ for all $u \in C$, then for the set of solutions $S:= \{w \in C: j(w) = j(\tu) \}$ and Algorithm~\ref{alg:PSGD_Hilbert} with step sizes satisfying \eqref{eq:RobbinsMonroStepSize}, it holds that
\begin{enumerate}
 \item $\{\lVert u_n - u \rVert^2\}$ converges a.s.\,for all $u \in S$,
 \item $\{j(u_n) \}$ converges a.s.\,and $\lim_{n \rightarrow \infty} j(u_n) = j(\tu)$,
 \item $\{u_n \}$ weakly converges a.s.\,to some $\bar{u} \in S.$
\end{enumerate}
\end{thm}

\begin{proof}
 For the first statement, let $u \in S$ be an arbitrary element in the solution set and let $g_n = G(u_n,\xi_n)$. Since $u \in C$, $\pi_C(u) = u.$ Thus using the nonexpansivity of the projection operator,
 \begin{equation}\label{eq:boundsequence}
 \begin{aligned}
\lVert u_{n+1} - u \rVert^2 &= \lVert \pi_C(u_n - \tau_n g_n) - \pi_C(u) \rVert^2 \\
&\leq  \lVert u_n - \tau_n g_n - u \rVert^2\\
& = \lVert u_n - u \rVert^2 - 2 \tau_n \langle u_n - u, g_n \rangle + \tau_n^2 \lVert g_n \rVert^2.
\end{aligned}
 \end{equation}
Since $\xi_n$ is independent from $\xi_1, \dots, \xi_{n-1}$, it follows that
\begin{equation}
\label{eq:gradient-martingale-to-mean}
 \E[\lVert g_n \rVert^2 | \mathcal{F}_n] = \E_\xi[\lVert G(u_n,\xi)\rVert^2],
\end{equation}
so 
\begin{equation}
 \label{eq:bound-variance-term}
 \E[\lVert g_n\rVert^2 | \mathcal{F}_n]  \leq M_1 + M_2 \lVert u_n \rVert^2 \leq M_3 + M_4 \lVert u_n - u \rVert^2
\end{equation}
for some constants $M_3$ and $M_4$, where we used  $\lVert u_n \rVert^2 \leq 2(\lVert u_n - u \rVert^2 + \lVert u \rVert^2)$. By assumption, $g_n = \nabla j(u_n) + w_n + r_n.$ Since $u_n$ and $ r_n$ are $\mathcal{F}_n$-measurable, it holds that \mbox{$\E[u_n |\mathcal{F}_n ] = u_n$} and $\E[r_n |\mathcal{F}_n ] = r_n$. Note as well that $\E[w_n|\mathcal{F}_n] = 0$ holds. Using \eqref{eq:boundsequence}, we therefore have
\begin{equation}\label{eq:finalinequality-partone}
\begin{aligned}
 &\E[\lVert u_{n+1} - u \rVert^2 | \mathcal{F}_n ] \\
 &\quad =\lVert u_n - u \rVert^2 - 2 \tau_n \E[\langle u_n - u, \nabla j(u_n) + w_n + r_n \rangle| \mathcal{F}_n] + \tau_n^2 \E[\lVert g_n \rVert^2 | \mathcal{F}_n]\\
 &\quad = \lVert u_n - u \rVert^2 - 2 \tau_n \langle u_n - u, \nabla j(u_n) + r_n \rangle + \tau_n^2 \E[\lVert g_n \rVert^2 | \mathcal{F}_n]\\
 &\quad \leq\lVert u_n - u \rVert^2 - 2 \tau_n (j(u_n) - j(u)) + 2\tau_n (\lVert u_n-u \rVert^2 + 1) \lVert r_n \rVert \\
 & \quad \qquad + \tau_n^2 (M_3 + M_4 \lVert u_n - u \rVert^2)\\
 &\quad = \lVert u_n - u \rVert^2(1+ 2 \tau_n \lVert r_n \rVert +\tau_n^2 M_4 ) - 2 \tau_n (j(u_n) - j(u)) + 2 \tau_n \lVert r_n \rVert + \tau_n^2 M_3 
\end{aligned}
\end{equation}
where in the inequality, we used convexity of $j$, the inequality \eqref{eq:bound-variance-term}, and the relation \mbox{$- 2\tau_n \langle u_n - u, r_n \rangle \leq 2 \tau_n (\lVert u_n - u \rVert^2 + 1) \lVert r_n \rVert $}. With
\begin{align*}
a_n &= 2 \tau_n \lVert r_n \rVert + \tau_n^2 M_4,\\
b_n & = 2 \tau_n \lVert r_n \rVert+ \tau_n^2 M_3, \\
c_n &= 2 \tau_n (j(u_n) - j(u)),
\end{align*}
observe that by Assumption~\ref{assumption:gradient}, $\sum_{n=1}^\infty a_n < \infty$ and $\sum_{n=1}^\infty b_n < \infty$ a.s. Clearly, $\{a_n\}$ and $\{b_n\}$ are nonnegative; $\{c_n\}$ is nonnegative by the fact that $u \in S$. Therefore by Lemma~\ref{lemma:Robbins-Siegmund}, the sequence $\{\lVert u_n - u \rVert^2\}$ converges a.s. Since $u \in S$ was arbitrary, the sequence must converge a.s.\,for all $u \in S$.

Now we show the second statement by verifying the conditions of Proposition~\ref{prop:technicalproposition}. By Lemma~\ref{lemma:Robbins-Siegmund}, with probability one it holds that
 \begin{equation}
 \label{eq:Robbins-Siegmund-result}
 \sum_{n=1}^\infty \tau_n (j(u_n) - j(u)) < \infty.
\end{equation}
For \eqref{eq:Robbins-Siegmund-result}, we just need that $\sum_{n=1}^\infty \tau_n \|r_n\| < \infty$ a.s. But we have assumed that $\sum_{n=1}^\infty \tau_n K_n < \infty$ and this may lead to a stronger result: taking the expectation on both sides of inequality (\ref{eq:finalinequality-partone}), and introducing $e_n:=\mathbb{E}[\| u_{n}-u\|^2]$, we get
\begin{align}
\label{eq:recursion}
e_{n+1}&\le e_n (1+2  \tau_n K_n + \tau_n^2 M_4 ) - 2 \mathbb{E}[\tau_n (j(u_n) - j(u))] + 2 \tau_n K_n + \tau_n^2 M_3,
\end{align}
from which we get using the deterministic version of Lemma~\ref{lemma:Robbins-Siegmund} that
\begin{equation*}
\mathbb{E}\big[ \sum_{n=1}^\infty \tau_n (j(u_n) - j(u)) \big] < \infty.
\end{equation*}
By convexity of $j$ in the first inequality, followed by the Cauchy-Schwarz inequality, and nonexpansivity of the projection operator in the third inequality,
 \begin{equation}\label{eq:convexity-inequality}
\begin{aligned}
  j(u_{n}) - j(u_{n+1}) & \leq \langle \nabla j(u_{n}), u_{n} - u_{n+1} \rangle \\
  & \leq \lVert \nabla j(u_{n}) \rVert \lVert u_{n+1} - u_n \rVert\\
  & = \lVert \nabla j(u_{n}) \rVert \lVert \pi_C(u_n - \tau_n g_n) - \pi_C(u_n) \rVert \\
  & \leq \lVert \nabla j(u_{n}) \rVert \tau_n \lVert  g_n \rVert.\\
\end{aligned}
 \end{equation}
Notice that 
$\nabla j(u_n) = \E[g_n|\mathcal{F}_n] - r_n = \E_\xi[G(u_n,\xi)] - r_n.$ Hence with probability one,
\begin{equation}
\label{eq:bounds-on-gradient-inequality}
\lVert \nabla j(u_n) \rVert \leq  \E_\xi[\lVert G(u_n,\xi) \rVert] + K_n \leq \sqrt{M_1} + \sqrt{M_2} \lVert u_n \rVert + K_n,
\end{equation}
since by Jensen's inequality, $\E_\xi[\lVert G(u_n,\xi) \rVert] \leq \sqrt{M_1 + M_2 \lVert u_n \rVert^2} \leq \sqrt{M_1} + \sqrt{M_2} \lVert u_n \rVert.$
Let $\sigma_M$ be the stopping time $\sigma_M = \inf \lbrace n: \lVert u_n \rVert > M \rbrace$ for $M \in \N$. On the set $\{\sigma_M > n\}$ we assign
$$\beta_{n}=j(u_n) - j(u),$$
while on $\{ \sigma_M \le n\}$ we assign 
$$\beta_{\sigma_M} = j(u_{\sigma_M})-j(u), \quad \beta_{\sigma_M+n}= \beta_{\sigma_M}, \quad n \geq 1.$$
Notice that if $\|u_n\| \le M$, then there exists by  \eqref{eq:bound-variance-term} and \eqref{eq:bounds-on-gradient-inequality} a $M^\prime$ such that
\begin{equation}\label{xyz}
\|\nabla j(u_n)\| \cdot {\E [ \|g_n\| | \mathcal{F}_n]} \le M^\prime(M^\prime + K_n).
\end{equation}
Now
$$\beta_n - \beta_{n+1} = \mathds{1}_{\sigma_M > n} (j(u_n) - j(u_{n+1}))$$
and therefore, taking the conditional  expectation on both sides, noticing that $\mathds{1}_{\sigma_M > n}$ is $\mathcal{F}_n$-measurable,  and considering \eqref{eq:convexity-inequality} and \eqref{xyz}, we get 
$$\beta_n - \E[ \beta_{n+1} | \mathcal{F}_n]  \leq M^\prime (M^\prime + K_n) \tau_n.$$
According to Proposition~\ref{prop:technicalproposition}, $\beta_n$ converges to 0 on the set $B_M :=\{\sigma_M = \infty\}$ and on this set, $\beta_n$ coincides with $j(u_n) - j(u)$.
Since $\|u_n - u\|$ converges a.s., $\|u_n\|$ is bounded in probability and therefore the probability of the set $B_M$ can be made arbitrarily close to 1 by choosing $M$ large. Since $\pP(\bigcup_{M=1}^\infty B_M) = 1$, we may infer that $j(u_n)-j(u)$ converges to 0 almost surely.

For the third statement, since $\{\lVert u_n - u \rVert^2\}$ converges a.s.~for all $u \in S$ by 1., it is bounded in probability, so there exists a weak accumulation point $\bar{u}$ of the sequence $\{ u_n\}$. The point $\bar{u}$ is random in general and in the following we argue pointwise for almost all $\omega \in \Omega$. Let $\{u_{n_k} \}$ be a subsequence of $\{ u_n\}$ such that $u_{n_k} \rightharpoonup \bar{u}$. Since $j$ is convex and continuous, it is weakly lower semicontinuous; cf.~\cite[p.~37]{Troeltzsch2009},
$$j(\bar{u}) \leq \lim_{k \rightarrow \infty} j(u_{n_k})  = j(\tu).$$
In particular, $\bar{u} \in S$. Since $\bar{u}$ was an arbitrary weak accumulation point, all weak accumulation points must belong to $S$. To show uniqueness, let $u_1, u_2 \in S$ be two distinct weak limits of $\{ u_n\}$, i.e. $u_{n_k} \rightharpoonup u_1$ and $u_{n_l} \rightharpoonup u_2$ and $u_1 \neq u_2.$ Then
\begin{align}
 \lVert u_{n_k} - u_2 \rVert^2 &= \lVert u_{n_k} - u_1 \rVert^2 +  \lVert u_1 - u_2 \rVert^2 + 2 \langle u_{n_k} - u_1, u_1 - u_2 \rangle,\label{eq:first-equality-part3}\\
\lVert u_{n_l} - u_1 \rVert^2 &=  \lVert u_{n_l} - u_2 \rVert^2 +  \lVert u_2 - u_1 \rVert^2 + 2  \langle u_{n_l} - u_2, u_2 - u_1 \rangle,\label{eq:second-equality-part3}
\end{align}
so by weak convergence of each subsequence, we combine \eqref{eq:first-equality-part3} and \eqref{eq:second-equality-part3} to obtain
\begin{align}
\lim_{k \rightarrow \infty} \lVert u_{n_k} - u_2 \rVert^2 - \lVert u_{n_k} - u_1 \rVert^2 = \lVert u_1 - u_2 \rVert^2,\label{eq:combination-equalities-part3-1}\\
 \lim_{l \rightarrow \infty} \lVert u_{n_l} - u_1 \rVert^2 - \lVert u_{n_l} - u_2 \rVert^2 = \lVert u_1 - u_2 \rVert^2. \label{eq:combination-equalities-part3-2}
\end{align}
By a.s.~convergence of the sequence $\{\lVert u_n - u \rVert^2\}$ for all $u \in S$, the limit of each subsequence is equal to the limit of the entire sequence with probability one, so $\lim_{k \rightarrow \infty} \lVert u_{n_k} - u_1 \rVert^2 = \lim_{n \rightarrow \infty}  \lVert u_{n} - u_1 \rVert^2 =: l_1$ and similarly $ \lim_{k \rightarrow \infty} \lVert u_{n_k} - u_2 \rVert^2 = \lim_{n \rightarrow \infty}  \lVert u_{n} - u_2 \rVert^2 =: l_2.$ 
Therefore \eqref{eq:combination-equalities-part3-1} and  \eqref{eq:combination-equalities-part3-2} imply
$$l_2 - l_1 = \lVert u_1 - u_2 \rVert^2 = l_1 - l_2,$$
meaning $\lVert u_1 - u_2 \rVert^2 = 0$ and thus the weak limits coincide. Therefore $\{ u_n\}$ is weakly convergent to a unique limit with probability one.
\end{proof}

We note that when $j$ is strongly convex, it is possible to establish almost sure strong convergence.
\begin{cor}\label{lemma:convergence:stronglyconvexcase}
With the same assumptions as in Lemma~\ref{lemma:convergenceofPSG}, assume that $j$ is additionally strongly convex. Then $\{u_n\}$ converges a.s.~to a unique minimum $\bar{u}$.
\end{cor}

\begin{proof}
By strong convexity, $j$ has a unique minimum $\bar{u}$, so $S = \{ \bar{u}\}$. By strong convexity, there exists a $\mu > 0$ such that
\begin{equation}
\label{eq:mu-strong-convexity}
j(u_n) - j(\bar{u}) \geq \langle \nabla j(\bar{u}), u_n -\bar{u} \rangle + \frac{\mu}{2} \lVert u_n - \bar{u} \rVert^2 \end{equation}
($j$ is $\mu$-strongly convex). Since $\langle \nabla j(\bar{u}), u_n - \bar{u} \rangle \geq 0$ by optimality of $\bar{u}$, $\lim_{n \rightarrow \infty} j(u_n) - j(\bar{u}) =0$ a.s.\,implies $\lim_{n \rightarrow \infty} \lVert u_n - \bar{u} \rVert = 0$ a.s.
\end{proof}

\subsection{Robust Step Size Rules and Efficiency}\label{subsec:stepsizerules}
Performance of Algorithm \ref{alg:PSGD_Hilbert} is dependent on an appropriate step size rule satisfying \eqref{eq:RobbinsMonroStepSize}. Here, we generalize appropriate choices as discussed in Nemirovski et al.~\cite{Nemirovski2009} to the case where $C$ may not be bounded.  For simplicity, we will observe the case where $g_n$ is unbiased, i.e. $r_n = 0$ for all $n$ and note that where bias is present, George and Powell \cite{George2006} have developed step size rules that minimize estimation error. 

Let $\bar{u}$ be an optimal solution of \eqref{eq:SAproblem} and set  $e_n = \E[\lVert u_n - \bar{u} \rVert^2]$.
If $j$ is $\mu$-strongly convex, \eqref{eq:boundsequence} implies using the inequality \mbox{$\E[\langle u_n - \bar{u}, \nabla j(u_n) \rangle] \geq \mu \E[\lVert u_n - \bar{u} \rVert^2]$} that
\begin{align*}
e_{n+1} &\leq e_n - 2 \tau_n \E[\langle u_n - \bar{u}, g_n \rangle] + \tau_n^2 \E[\lVert g_n \rVert^2]\\
& \leq e_n (1 - 2  \mu\tau_n)+ \tau_n^2 (M_1 + M_2 \E[\lVert u_n \rVert^2])\\
& \leq e_n (1-2 \mu\tau_n +2\tau_n^2 M_2) + \tau_n^2 (M_1 + 2M_2 \lVert \bar{u} \rVert^2),  
 \end{align*}
where we also used Assumption~\ref{assumption:gradient} and $\lVert u_n \rVert^2 \leq 2\lVert u_n - \bar{u} \rVert^2 + 2\lVert \bar{u} \rVert^2.$ 
Note that for a recursion of the form $ e_{n+1} \leq e_n (1 - c_1/(n+\nu) +c_2/(n+\nu)^2)+c_3/(n+\nu)^2$, where $e_1, c_2, c_3 \geq 0$ and $c_1 > 1$, it holds that 
$$e_n \leq \frac{K}{n+\nu},$$ 
where $K:=(c_3 +e_1 c_2)/(c_1 -1)$ and $\nu +1= (c_3 + e_1 c_2)/(e_1(c_1-1)),$
which can be proven by induction; see Lemma~\ref{lemma:recursion-statement} in the Appendix. Therefore with the step size rule
\begin{equation}\label{eq:stepsizerule-stronglyconvex}
\tau_n = \frac{\theta}{n+\nu}
\end{equation}
we get with $e_1=\lVert u_1-\bar{u}\rVert$, $c_1 = 2 \mu \theta$, $c_2 = 2 \theta^2 M_2$, $c_3 = \theta^2(M_1 + 2 M_2 \lVert \bar{u} \rVert^2)$ the following efficiency estimate:
\begin{equation}\label{eq:iterationerrorbounds}
 \E[\lVert u_n - \bar{u} \rVert] \leq \sqrt{\frac{K}{n+\nu}}.
\end{equation}
If we additionally have that $\nabla j(u)$ is Lipschitz continuous with constant $L>0$ and $\bar{u}$ is an interior point of the admissible set $C$, then it holds that
$$j(u_n) \leq j(\bar{u}) + \frac{L}{2} \lVert u_n - \bar{u} \rVert^2,$$
so the expected error can also be bounded as follows:
\begin{equation}\label{eq:objectiveerrorbounds}
\E[j(u_n)-j(\bar{u})] \leq \frac{LK}{2(n+\nu)}. 
\end{equation}
We observe that our estimates have the same order as those of \cite{Nemirovski2009} where $M_2 = 0$ (no growth term) and $M_1 = M^2$ (uniform bound over $C$). Note that $\lVert \bar{u} \rVert$, which is generally unknown in the constant $K$, can be further estimated by $\lVert \bar{u} \rVert \leq \lVert \bar{u} - u_1 \rVert + \lVert u_1 \rVert.$ Finally, we remark that the step size rule \eqref{eq:stepsizerule-stronglyconvex} depends on a good estimate of the parameter $\mu.$

In the general convex case, or where a good estimate for $\mu$ does not exist, step sizes of the form $\tau_n = \theta/(n+\nu)$ may be too small for efficient convergence. An idea is to use averaging of iterates to suppress noise, combined with larger step sizes, which was developed in \cite{Polyak1992}. From \eqref{eq:boundsequence} we get
\begin{equation}
\label{eq:recursion-convex}
\begin{aligned}
e_{n+1} &\leq e_n - 2 \tau_n \E[\langle u_n - \bar{u}, g_n\rangle] + \tau_n^2 \E[\lVert g_n \rVert^2]\\
& \leq e_n(1+2\tau_n^2 M_2) - 2\tau_n \E[j(u_n)-j(\bar{u})] + \tau_n^2 (M_1 + 2M_2 \lVert \bar{u}\rVert^2).
\end{aligned}
 \end{equation}  
Rearranging \eqref{eq:recursion-convex} and summing over $1 \leq i \leq N$ on both sides,
\begin{equation}
\begin{aligned}
\label{eq:estimate-jerror-convex}
\sum_{n=i}^N \tau_n \E[j(u_n)-j(\bar{u})] &\leq \sum_{n=i}^N \frac{e_n}{2} (1+ 2 \tau_n^2 M_2) - \frac{e_{n+1}}{2} + \frac{\tau_n^2 M_1 }{2} +  \tau_n^2 M_2\lVert \bar{u}\rVert^2\\ 
&\leq \frac{e_i}{2} + \frac{1}{2} \sum_{n=i}^N   \tau_n^2 (2M_2 e_n + {M_1} +  2M_2  \lVert \bar{u} \rVert^2).
\end{aligned}
\end{equation}
We define $\gamma_n:=\tau_n/(\sum_{l=i}^N \tau_l)$ and the average of the iterates $i$ to $N$ as
\begin{equation*}\label{eq:averagediterates}
\tilde{u}_i^N = \sum_{n=i}^N \gamma_n u_n.
\end{equation*}
By convexity of $j$, we have $j(\tu_i^N) \leq \sum_{n=i}^N \gamma_n j(u_n)$ so by \eqref{eq:estimate-jerror-convex}
\begin{equation}
\label{eq:basic-efficiency-estimate-convex}
 \E[j(\tu_i^N)-j(\bar{u})] \leq \frac{e_i + \sum_{n=i}^N \tau_n^2 (2M_2 e_n + {M_1} +  2M_2  \lVert \bar{u} \rVert^2)}{2 \sum_{n=i}^N \tau_n}.
\end{equation}
We first summarize from \cite{Nemirovski2009} the case where the stochastic gradient is uniformly bounded over a bounded set $C$, i.e.~there exists a $M>0$ such that $\E[\lVert G(u,\xi) \rVert^2] \leq M$ for all $u \in C$. Then \eqref{eq:basic-efficiency-estimate-convex} reduces to 
$$\E[j(\tu_i^N)-j(\bar{u})] \leq \frac{e_i + M \sum_{n=i}^N \tau_n^2  }{2 \sum_{n=i}^N \tau_n}.$$
In the case where $C$ is bounded, using $D_C:=\sup_{u \in C} \lVert u - u_1\rVert$, it follows that $e_1 \leq D_C^2$ and $e_i\leq 4 D_C^2$ for $1<i\leq N$. With the constant stepsize policy for a fixed number of iterations $N$ and $n=1, \dots, N$,
\begin{equation}\label{eq:stepsizerule-convex-fixed}
 \tau_n = \frac{D_C}{\sqrt{M N}}
\end{equation}
we get after plugging \eqref{eq:stepsizerule-convex-fixed} into \eqref{eq:estimate-jerror-convex2} the efficiency estimate
\begin{equation}
\quad \E[j(\tilde{u}_1^N) - j(\bar{u})] \leq \frac{D_C \sqrt{M}}{\sqrt{N}}.
\end{equation}
Alternatively, one can work with the (nonconstant) step size policy for a constant $\theta > 0$
\begin{equation}\label{eq:stepsizerule-convex}
 \tau_n = \frac{\theta D_C}{\sqrt{M n}}.
\end{equation}
Then for $k=\lceil r N \rceil$ and a fixed $r \in (0,1)$, we get the efficiency estimate 
\begin{equation}\label{eq:efficiency-convexcase}
 \quad \E[j(\tilde{u}_k^N) - j(\bar{u})] \leq C(r) \max\{\theta, \theta^{-1}\}\frac{D_C \sqrt{M}}{\sqrt{N}},
\end{equation}
where $C(r)$ is a constant depending on $r$.

In the case where $C$ is generally unbounded, we need to obtain efficiency estimates in a different way. For this, we assume that at least the solution set is bounded, and denote by $D_S$ a constant depending on the solution set $S$ such that
$$\sup_{u \in S} \|u_1 - u\| \le D_S.$$
Then from \eqref{eq:recursion-convex}, we get for $n > 1$
\begin{align*}
e_{n+1} &\le e_1 \prod_{k=1}^n (1 + 2 \tau_k^2 M_2) + \sum_{k=1}^n \tau_k^2 (M_1 + 2 M_2 \lVert \bar{u} \rVert^2) \prod_{l=k+1}^n (1+ 2\tau^2_l M_2) =:Q_n
\end{align*}
for some $Q_n$ depending on $M_1, M_2$, $D_S$, and $\tau_k$, for $k=1, \dots, n.$ We consider the step sizes $\tau_n = \theta/n^{\gamma}$ for $1/2 < \gamma < 1$. Notice that then $Q:=\sup_{n} Q_n < \infty$,
$$\sum_{k=1}^N k^{-2\gamma} \leq \frac{2\gamma}{2 \gamma -1},$$
and
$$\sum_{k=1}^N k^{-\gamma} \geq \frac{1}{1-\gamma} (N+1)^{1-\gamma}.$$
Using these estimates we get from \eqref{eq:basic-efficiency-estimate-convex} and $R:=2M_2 Q + M_1 +2 M_2 \lVert \bar{u} \rVert^2$ that
\begin{equation}
\label{eq:estimate-jerror-convex2}
\E[j(\tu_1^N)-j(\bar{u})] \leq \frac{(1-\gamma) D_S^2 + (2 R \theta^2 \gamma (1-\gamma))/(2\gamma-1)}{2 \theta (N+1)^{1-\gamma}} .
\end{equation}
Notice that the speed of convergence comes close to the order $N^{-1/2}$ (as in the bounded case) if $\gamma$ is chosen close to $1/2$.

\section{Application to PDE Constrained Optimization under Uncertainty}\label{sec:ModelProblem}
We now will demonstrate application of Algorithm~\ref{alg:PSGD_Hilbert} to a model problem, the optimal control of a stationary heat source, which is subject to uncertain material parameters. Proofs, where omitted, are to be found in the Appendix. In the following, the inner product between two vectors $v, w \in \R^d$ is denoted by $v \cdot w = \sum_{i=1}^d v_i w_i$. For a function $f:\R^d \rightarrow \R$, let $\nabla f(x) = ({\partial f(x)}/{\partial x_1}, \dots, {\partial f(x)}/{\partial x_d})^\top$ denote the gradient and for $g: \R^d \rightarrow \R^d$, let $\nabla \cdot g(x) = {\partial g_1(x)}/{\partial x_1} + \cdots + {\partial g_d(x)}/{\partial x_d}$ denote the divergence. We define the Sobolev space $H^1(D)$ = \{$u\in L^2(D)$: ${\partial u}/{\partial x_i} \in L^2(D)$, $i =1, \dots, d$\} and the closure of $C_0^\infty(D)$ in $H^1(D)$ by $H_0^1(D)$.  The space $H_0^1(D)$ is a Hilbert space with inner product defined as $\langle f,g\rangle_{H_0^1(D)} = \int_D f(x) g(x) \D x + \int_D \nabla f(x) \cdot \nabla g(x) \D x.$ We also use the notation $|f|_{H_0^1(D)}^2 := \int_D |\nabla f(x)|^2 \D x$ for the $H_0^1(D)$-seminorm. 

Let $D \subset \R^d$ for $d=2,3$ be a bounded Lipschitz domain\footnote{A Lipschitz domain $D$ requires that for every point $x \in \partial D$, there exists a neighborhood in the boundary $\partial D$ that can be expressed as the graph of a Lipschitz-continuous function; see \cite[p.21]{Troeltzsch2009} for a technical definition. Polygonal domains in $\R^2$ and polyhedra in $\R^3$ automatically satisfy this assumption.}. Denote the probability space with $(\Omega, \mathcal{F}, \mathbb{P})$ and let $a: D \times \Omega \rightarrow \R$ be a random field representing conductivity on the domain. A realization of the field $a$ is denoted by $a(\cdot,\omega)$ for $\omega \in \Omega$. Temperature $y$ is a random function controlled by the deterministic source density $u$. The factor $\lambda$ is a measure of the energy costs related to the control $u$. The goal is to find a $u$ with corresponding $y$ that, in expectation, best approximates a deterministic target temperature $y_D$ with minimal cost. Mathematically, the problem is given by
\begin{equation}\label{eq:stationaryheatsourceproblemrandom}
 \begin{aligned}
   \min_{u \in \U} \quad \Big\lbrace j(u):= \E[ J(u, \omega) ] &:= \E \left[\frac{1}{2} \lVert y  - y_D\rVert_{L^2(D)}^2 \right] + \frac{\lambda}{2} \lVert u \rVert_{L^2(D)}^2 \Big\rbrace
   \\
   \text{s.t.} \quad -   \nabla \cdot  (a(x,\omega) \nabla y(x,\omega)) &= u(x), \qquad (x,\omega) \in D \times \Omega, \\
   y(x,\omega) &= 0, \phantom{tex}\qquad (x,\omega) \in \partial D \times \Omega,\\
    \U:= \{ u \in L^2(D):  & \,u_a(x)  \leq u(x) \leq u_b(x)\,\,\text{ a.e. } x\in D\}.
 \end{aligned}
\end{equation}
The admissible set $\U$ is clearly nonempty, bounded, convex, and closed. Additionally, $j(u)$ is convex by linearity of the mapping $T_\omega: u \mapsto y$; see Lemma~\ref{lemma:existenceuniquenessoptimalcontrol-stochastic}. Randomness in the conductivity is assumed to be finite in the sense that there exist $a_{\min}, a_{\max}$ such that for all $(x,\omega) \in D \times \Omega,$
\begin{equation}\label{eq:boundsona}
0 < a_{\min} < a(x,\omega) < a_{\max} < \infty.\end{equation}
Such restrictions can be weakened to allow for log-normal random fields; see \cite{Lord2014}. We recall properties of the weak solutions to the PDE constraint in \eqref{eq:stationaryheatsourceproblemrandom}. 

\begin{lemma}\label{lemma:poissonrandom}
Let $u \in L^2(D)$ and $a(\cdot,\omega)$ satisfy \eqref{eq:boundsona} for all $x \in D$. Then there exists a unique $y(\cdot, \omega) \in H_0^1(D)$ that satisfies
\begin{equation}\label{eq:Poissonrandom}
\int_D a(x,\omega) \nabla y(x,\omega) \cdot \nabla v(x) \D x = \int_D u(x) v(x) \D x \quad \forall v \in H_0^1(D).
\end{equation}
Moreover, there exists a constant $C_1 > 0$ such that
\begin{equation} \label{eq:Aprioribounds_Poisson_stochastic}
 \lVert y(\cdot, \omega)\rVert_{L^2(D)} \leq C_1 \lVert u \rVert_{L^2(D)}.
\end{equation}
\end{lemma}

Existence and uniqueness of problem \eqref{eq:stationaryheatsourceproblemrandom} for the deterministic case was shown in \cite{Troeltzsch2009}. For the random case, \cite{Hou2011} already presented a proof of existence for the unconstrained case. We present a proof for the constrained case and show uniqueness if $\lambda = 0$.

\begin{lemma}\label{lemma:existenceuniquenessoptimalcontrol-stochastic}
Let $u \in L^2(D)$, $a(\cdot, \cdot)$ satisfy \eqref{eq:boundsona} for all $(x,\omega) \in D \times \Omega$, and $y_D \in L^2(D)$.
Assume $\lambda \geq 0,$ then there exists a solution $\bar{u}$ to \eqref{eq:stationaryheatsourceproblemrandom}. If $\lambda >0$, the solution is unique.
\end{lemma}

\begin{proof}
 $j$ is bounded from below, since for all $u \in L^2(D)$, $J(u,\omega) \geq 0$ a.s. Therefore there exists an infimum
$$\bar{j} := \inf_{u \in \U} j(u) \geq 0.$$
For a minimizing sequence \mbox{$\{u_n \}\subset \U$} such that $\lim_{n \rightarrow \infty} j(u_n) = j(\bar{u})$, there exists a subsequence $\{u_{n_k} \}$ such that $u_{n_k} \rightharpoonup \bar{u}$, since sequences in $\U$, a convex, closed and bounded subset of $L^2(D)$, are weakly sequentially compact.

By \eqref{eq:Aprioribounds_Poisson_stochastic} and the assumptions on $y_D$ and $u$, $J(\cdot, \omega)$
is bounded for almost every $\omega \in \Omega$. The mapping $T_\omega: L^2(D) \rightarrow H_0^1(D), u \mapsto y$ for each $\omega \in \Omega$ is well-defined by Lemma \ref{lemma:poissonrandom}, and is clearly linear. Thus $J(u,\omega)= \tfrac{1}{2} \lVert T_\omega u - y_D \rVert_{L^2(D)}^2 + \tfrac{\lambda}{2} \lVert u \rVert_{L^2(D)}^2$ is convex in $u$. 
By monotonicity of the expectation operator, the function $j$ is convex, and therefore weakly lower-semicontinuous, i.e.
$$j(\bar{u}) \leq \liminf_{k \rightarrow \infty} j(u_{n_k})=\bar{j}.$$
Since $\bar{u} \in \U,$ $j(\bar{u})$ cannot be smaller than $\bar{j}$. Therefore $j(\bar{u})=\bar{j}.$

For uniqueness, we note that when $\lambda \neq 0$, $j$ is a strongly convex function and therefore strictly convex. If there were two optima $\bar{u} \neq \bar{v},$ then  $j(\tfrac{1}{2}(\bar{u}+\bar{v})) < \tfrac{1}{2}j(\bar{u})+\tfrac{1}{2}j(\bar{v}) = j(\bar{u}),$
which is a contradiction by optimality of $\bar{u}.$
\end{proof}

%

\begin{prop}\label{prop:gradient}
For $\omega \in \Omega$, the stochastic gradient $\nabla_u J(u,\omega)$ for problem \eqref{eq:stationaryheatsourceproblemrandom} is given by
$$\nabla_u J(u,\omega) = \lambda u - p(\cdot,\omega),$$
where $p(\cdot,\omega) \in H_0^1(D)$ solves the PDE
\begin{equation}\label{eq:Poissonrandom-adjoint}
\int_D   a(x,\omega)\nabla v(x) \cdot \nabla p(x,\omega) \D x = \int_D (y_D(x) - y(x,\omega)) v(x) \D x \quad \forall v \in H_0^1(D).
\end{equation}
\end{prop}

Algorithm~\ref{alg:PSGD_Hilbert} applied to \eqref{eq:stationaryheatsourceproblemrandom} is therefore:
\begin{algorithm}[H]
 \caption{PSG for Random Stationary Heat Problem}
\label{alg:PSGD_StationaryHeat}
 \begin{algorithmic}[0]
\State \textbf{Initialization:} $u_1 \in L^2(D)$.
\For{$n= 1,2,\dots$}
\State Generate random $a(\cdot, \omega_n)$, independent from previous observations, and $\tau_n >0$
\State $y_n \gets$ solution to \eqref{eq:Poissonrandom} with $a(\cdot,\omega) = a(\cdot,\omega_n)$
\State $p_n \gets$ solution to \eqref{eq:Poissonrandom-adjoint} with $y=y_n$ and $a(\cdot,\omega) = a(\cdot,\omega_n)$
\State $G(u_n,\omega_n) := \lambda u_n - p_n$
\State $u_{n+1} := \pi_{\U}(u_n - \tau_n G(u_n,\omega_n))$
\EndFor
 \end{algorithmic}
\end{algorithm}

To prove convergence of Algorithm~\ref{alg:PSGD_StationaryHeat}, we need the following result.
\begin{lemma}\label{lemma:poissonrandom-adjoint}
Let $y(\cdot,\omega), y_D \in L^2(D)$. Then there exists a unique $p(\cdot, \omega) \in H_0^1(D)$ that satisfies \eqref{eq:Poissonrandom-adjoint}.
Moreover, there exists a constant $C_2>0$ such that for almost every $\omega \in \Omega$
\begin{equation} \label{eq:Aprioribounds_Poisson_stochastic_adjoint}
 \lVert p(\cdot, \omega)\rVert_{L^2(D)} \leq C_2 \lVert y_D - y(\cdot,\omega) \rVert_{L^2(D)}. \\
\end{equation}
\end{lemma}

\begin{thm}
Suppose that $a(\cdot, \cdot)$ satisfies \eqref{eq:boundsona}. If step sizes are chosen satisfying \eqref{eq:RobbinsMonroStepSize}, then for \eqref{eq:stationaryheatsourceproblemrandom}, the sequence $\{ u_n\}$ generated by Algorithm~\ref{alg:PSGD_StationaryHeat}
\begin{enumerate}
 \item converges strongly a.s.~to the unique optimum $\bar{u}$, if $\lambda \neq 0$.
 \item converges weakly a.s.~to a point in the set $S=\{u \in C: j(u)\leq j(w) \, \, \forall w \in C\}$ if $\lambda = 0$. 
\end{enumerate}
\end{thm}

\begin{proof}
We will verify the requirements of Lemma~\ref{lemma:convergenceofPSG} with $G(u,\omega) = \nabla_u J(u,\omega)$. In the proof for Lemma~\ref{lemma:existenceuniquenessoptimalcontrol-stochastic}, we already showed that $u \mapsto J(u,\xi(\omega))$ is convex; $L^2$-Fr\'echet differentiability is clear, therefore Assumption~\ref{assumption:j} is clearly satisfied. With the bounds \eqref{eq:Aprioribounds_Poisson_stochastic} and \eqref{eq:Aprioribounds_Poisson_stochastic_adjoint}, 
\begin{equation}\label{eq:boundsforgradient-apriori}
\begin{aligned}
 \lVert G(u,\omega) \rVert_{L^2(D)} & \leq \lambda \lVert  u  \rVert_{L^2(D)} + \lVert p(\cdot,\omega) \rVert_{L^2(D)}\\
 & \leq \lambda \lVert  u  \rVert_{L^2(D)} + C_2 ( \lVert y_D \rVert_{L^2(D)} + \lVert y(\cdot,\omega) \rVert_{L^2(D)})\\
 & \leq \lambda  \lVert  u  \rVert_{L^2(D)} + C_2 ( \lVert y_D \rVert_{L^2(D)} + C_1\lVert u \rVert_{L^2(D)}).
\end{aligned}
\end{equation}
Thus there exist constants $M_1$, $M_2$ such that $\lVert G(u,\omega) \rVert_{L^2(D)} \leq M_1 + M_2 \lVert u \rVert_{L^2(D)}^2$.
Since $\U$ is bounded, $\lVert G(u,\omega)\rVert_{L^2(D)}$ is even dominated by a deterministic constant $M$ with probability one. By Lebesgue's dominated convergence theorem,
$$\nabla j(u) = \frac{\D}{\D u} \E [J(u,\xi)] = \E \left[\nabla_u J(u,\xi) \right] = \E[G(u,\xi)].$$
In particular, $G(u,\omega) = \nabla j(u) + w$ for a random variable satisfying $\E[w] = 0.$ There is no bias term if the stochastic gradient is chosen such that $G(u,\omega) = \nabla_u J(u, \omega).$ Therefore all conditions of Assumption~\ref{assumption:gradient} are satisfied.

The set of solutions $S$ is nonempty by Lemma~\ref{lemma:existenceuniquenessoptimalcontrol-stochastic}. Hence we can conclude the following.

 \begin{enumerate}
 \item If $\lambda \neq 0$, $j$ is strongly convex and therefore by Lemma~\ref{lemma:convergence:stronglyconvexcase} $\{u_n\}$ strongly converges a.s.~to a unique minimum $\bar{u}$.
 \item If $\lambda = 0$, $j$ is convex and therefore by Lemma~\ref{lemma:convergenceofPSG}, $\{ u_n\}$ weakly converges a.s.~to a point in the solution set $S$.
 \end{enumerate}
\end{proof}

\section{Numerical Experiments}\label{sec:Numerics}
To demonstrate Algorithm~\ref{alg:PSGD_StationaryHeat}, let the domain be given by $D=[0,1]\times[0,1]$ and $\U = \{ u \in L^2(D) | -1 \leq u(x) \leq 1  \quad \forall x \in D\}.$ In this case, the projection $\pi_{\U}$ can be computed pointwise using the formula $\pi_{\U}(u) = \min \{1, \{\max\{-1, u \} \}.$ For the sake of illustration, assume that the material parameter satisfies $a(x,\omega) = a(\omega) \in \R$ for all $x\in D$.

For simulations, a finite element uniform triangulation of piecewise linear elements and with $3990$ nodes ($h_{\min} \approx 0.013)$ was used for $D$.
Simulations were run on FEniCS \cite{Alnes2015} on a laptop with Intel Core i7 Processor (8 x 2.6 GHz) with 16 GB RAM.

\paragraph{Strongly convex case}
An example was constructed for $\lambda > 0$ where the optimum of \eqref{eq:stationaryheatsourceproblemrandom} is known in the deterministic case.  We choose $\bar{p}(x) = -\sin(2 \pi x_1) \sin(2 \pi x_2),$ which in particular is equal to zero on the boundary of $D$. An optimum $\bar{u}$ must satisfy $\langle \lambda \bar{u} - \bar{p}, w - \bar{u} \rangle_{L^2(D)} \geq 0$ for all $w \in \U.$ Thus $\bar{u} = \pi_{\U}(\tfrac{1}{\lambda} \bar{p}).$ We have $\bar{y} = -\tfrac{1}{\bar{a} 8 \pi^2 \lambda}\sin(2\pi x_1)\sin(2 \pi x_2),$ which satisfies the strong form of \eqref{eq:Poissonrandom}. Finally, we have
$y_D(x) = -\left( \bar{a}8 \pi^2 + \frac{1}{\bar{a}8 \pi^2 \lambda} \right)\sin(2 \pi x_1) \sin(2 \pi x_2),$ which satisfies the strong form of \eqref{eq:Poissonrandom-adjoint}.

For the experiments, we chose $\bar{a}=2$ and $\lambda=2$, resulting in the target temperature $y_D(x)=-\left( 16 \pi^2 + \frac{1}{32 \pi^2} \right)\sin(2 \pi x_1) \sin(2 \pi x_2)$. Values for $a(\omega)$ are chosen randomly from a truncated normal distribution defined on the interval $[0.5,3.5]$ with mean $2$ and standard deviation $\sigma = 0.25;$ these were chosen to satisfy the bounds \eqref{eq:boundsona}. We use the step size rule \eqref{eq:stepsizerule-stronglyconvex} with $\theta = \tfrac{1}{3}$, where it is noted that an optimal bound for the strong convexity parameter $\mu$ is equal to $\lambda.$

Results of the simulation are in Figure~\ref{fig:objfuncstationaryheat}. The function $u_{N}$ as expected approximates the form of the deterministic optimum $\bar{u}= -\frac{1}{2}\sin(2 \pi x_1) \sin(2 \pi x_2)$. To investigate convergence behavior, the reference solution $\tilde{u}$ is obtained by running the algorithm for $N=10,000$ steps on a finer mesh ($15,681$ nodes, $h_{\min} \approx 6.6 \cdot 10^{-3}$). To compute objective function values, we use $\hat{j}(u_n) = \frac{1}{2} \lVert \hat{y}_n - y_D \rVert_{L^2(D)}^2 + \frac{\lambda}{2} \lVert u_n \rVert_{L^2(D)}$ as an estimate of the objective function in problem \eqref{eq:stationaryheatsourceproblemrandom}. Note that $\hat{y}_n$ corresponds to a single random solution of the problem \eqref{eq:Poissonrandom} with $u=u_n$ and $a(\omega) = a(\omega_n)$. The error of objective function values $\hat{j}(u_n) - \hat{j}(\tilde{u})$ as a function of iteration number is plotted on a log/log scale to demonstrate convergence behavior of $\mathcal{O}(n^{-1.00})$, which is consistent with the expected error from \eqref{eq:objectiveerrorbounds}. The error of iterates $\lVert u_n - \tilde{u} \rVert_{L^2(D)}$ is similarly plotted to display convergence of the form $\mathcal{O}(n^{-0.78})$, which is better than the expected convergence from \eqref{eq:iterationerrorbounds}.

\begin{figure}
\centering
 \begin{subfigure}{.45\textwidth}
   \includegraphics[width=\linewidth]{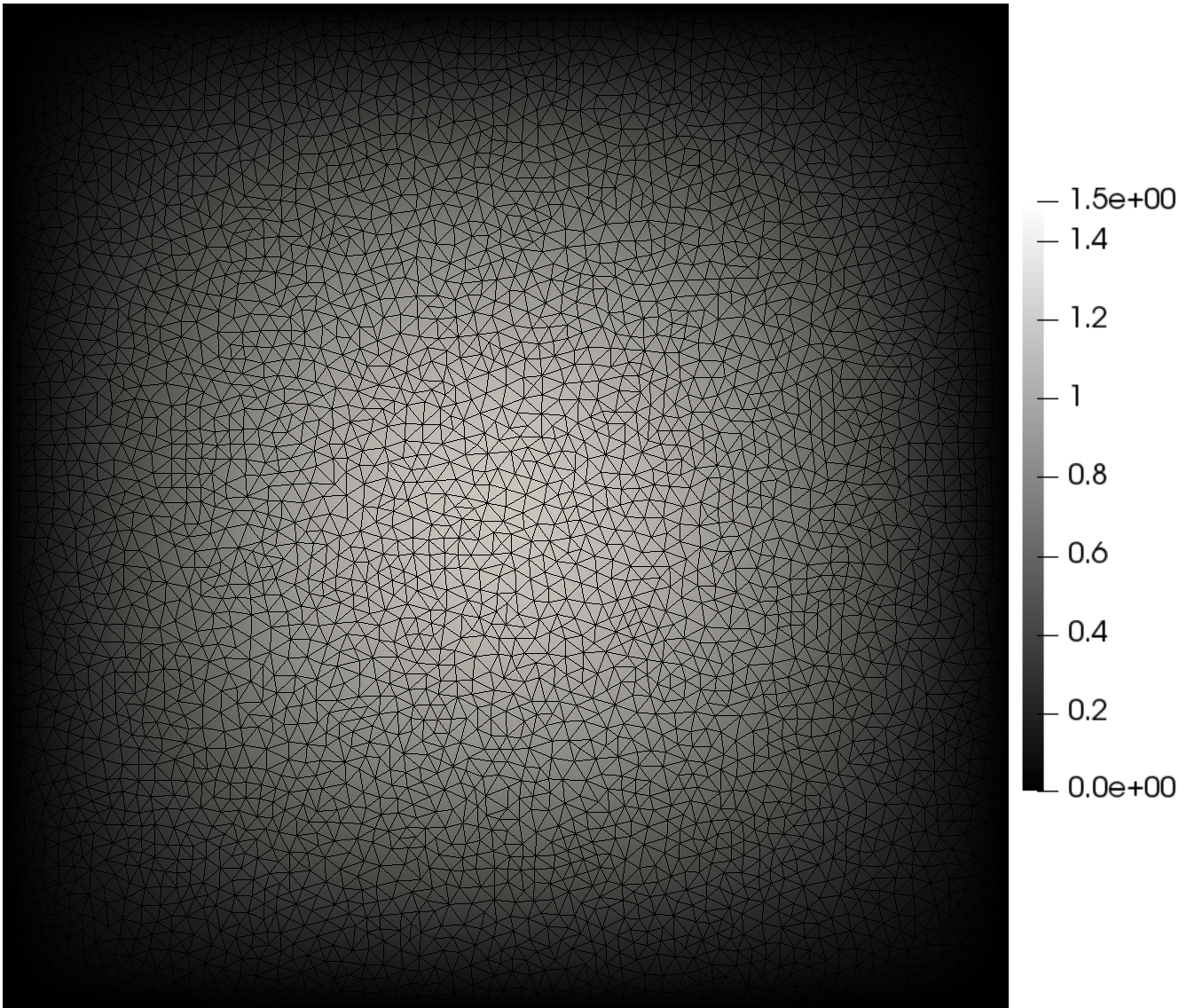}
  \caption{$u_{1}=\frac{3}{2}\sin(\pi x) \sin(\pi y)$}
 \end{subfigure}
   \begin{subfigure}{.45\textwidth}
   \includegraphics[width=\linewidth]{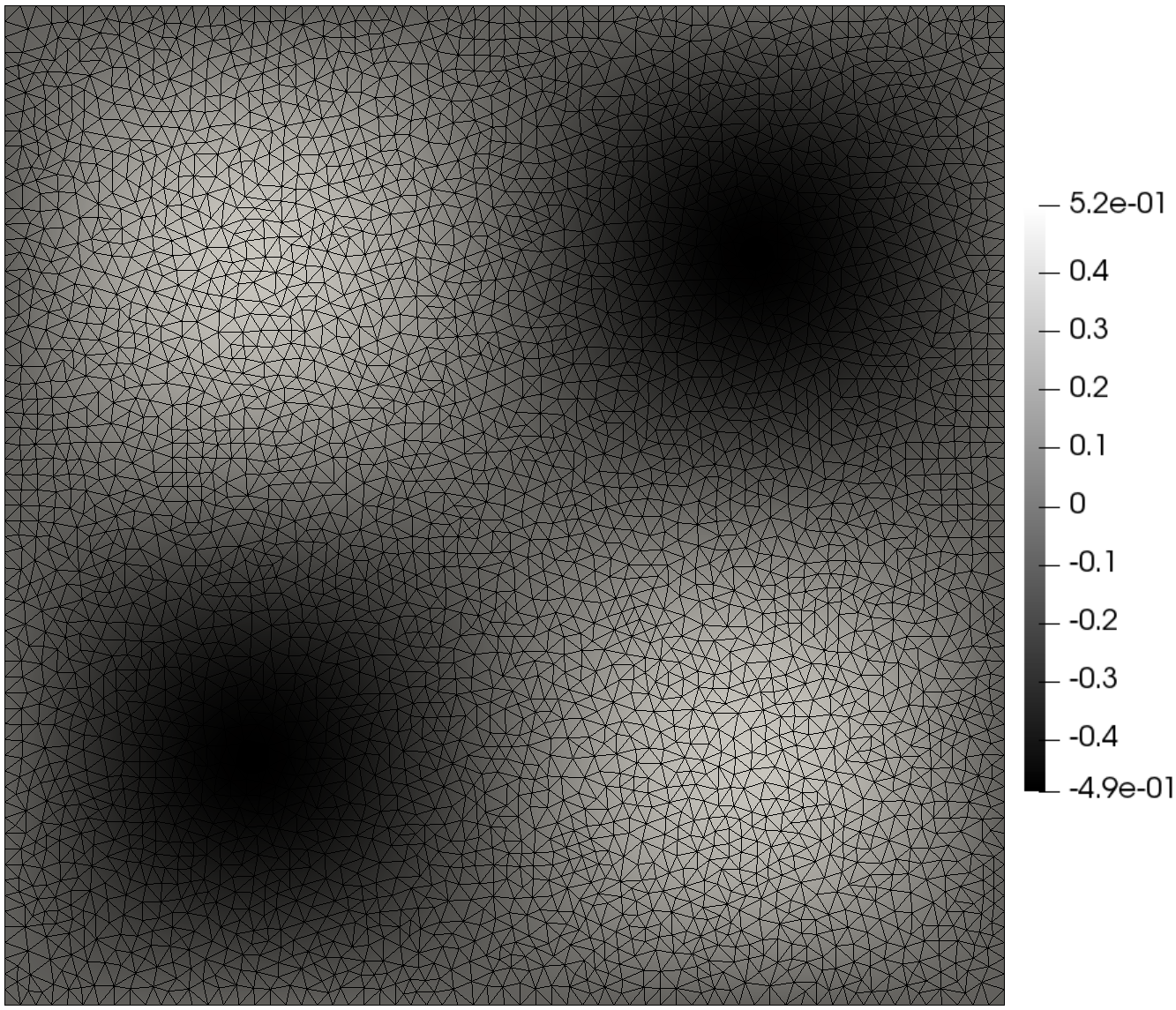}
  \caption{$u_{N}$}
 \end{subfigure}

 \begin{subfigure}{.45\textwidth}
   \includegraphics[width=\linewidth]{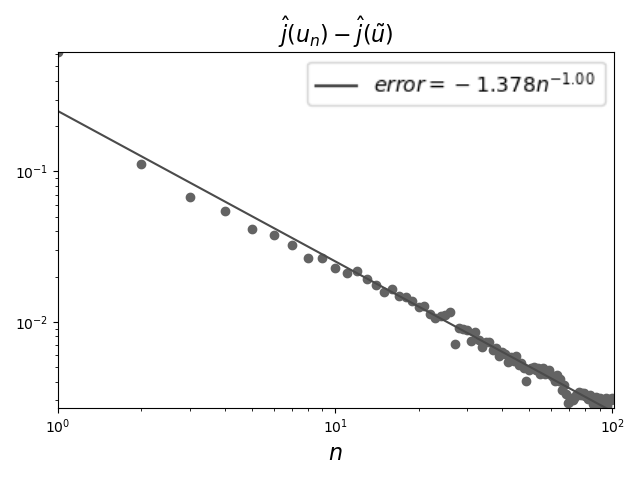}
  \caption{Errors in the objective function value}
 \end{subfigure}
   \begin{subfigure}{.45\textwidth}
   \includegraphics[width=\linewidth]{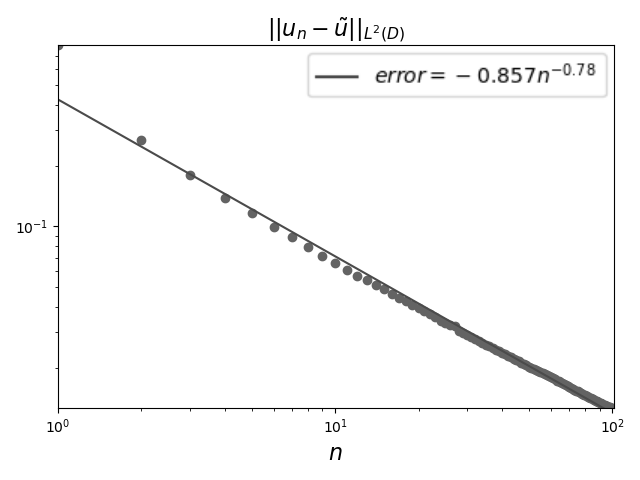}
  \caption{Errors in the control}
 \end{subfigure}
   \caption{Strongly convex experiment with $N=100$ iterations.}
\label{fig:objfuncstationaryheat}
 \end{figure}

\paragraph{General convex case}
For the general convex case, we set $\lambda=0$ and the simulate the following modified problem. Note that the introduction of $e_D$ is needed only to analytically generate a deterministic solution to the problem; clearly, the objective function is convex.
\begin{align*}
 \min_{u \in \U} \Big\lbrace \E[J(u,\omega)] &= \E\left[\frac{1}{2}\lVert y-y_D \rVert_{L^2(D)}^2 \right] \Big\rbrace\\
 \text{s.t.} \quad -\nabla \cdot (a(\omega) \nabla y(x,\omega)) &= u(x)+e_D(x), \quad (x,\omega) \in D \times \Omega\\
 y(x,\omega) &=0,\phantom{(x)+e_D(x),}  \quad(x,\omega) \in \partial D \times \Omega.
\end{align*}
An example was constructed where the deterministic optimum is known as in \cite{Troeltzsch2009}. We choose $p(x)=\frac{1}{8\pi^2}\sin(2\pi x_1) \sin(2\pi x_2)$ and $\bar{y}(x)=\sin(\pi x_1) \sin(\pi x_2)$, and note that the deterministic optimum is the bang bang solution $\bar{u}(x) = \text{sign}(p(x)).$ For the experiments, we choose $y_D(x) = \sin(\pi x_1) \sin(\pi x_2)+2\sin(2 \pi x_1) \sin(2\pi x_2)$ and $e_D(x) = 4 \pi^2 \sin(\pi x_1) \sin(\pi x_2) - \text{sign}( \sin(2\pi x_1) \sin(2\pi x_2))$. We use the same distribution for $a(\omega)$ as in the strongly convex case. We employ averaging of the iterates as in \eqref{eq:averagediterates} with $i=1$, i.e.\,$\tilde{u}_1^N = \sum_{k=1}^N \gamma_k u_k$ and $\gamma_k = \tau_k/(\sum_{l=1}^N \tau_l)$. The robust step size rule \eqref{eq:stepsizerule-convex} with $\theta = 500$ is used, which was obtained after tuning. Note that $D_S = 1$ if $u_1(x,y)=0$ (in the center of the admissible set). From \eqref{eq:boundsforgradient-apriori}, noting that the right hand side of the PDE constraint is $u+e_D$ instead of $u$, we have  $\lVert G(u,\omega) \rVert_{L^2(D)} \leq C_2(\lVert y_D \rVert_{L^2(D)} + C_1(\lVert u \rVert_{L^2(D)} +\lVert e_D\rVert_{L^2(D)})) \leq 3.9.$  For the bound, we used that $C_1 = C_2 = C_p^2/a_{\min}$ by the proofs for Lemma~\ref{lemma:poissonrandom} and Lemma~\ref{lemma:poissonrandom-adjoint} with $a_{\min}=0.5$ and the Poincar\'e constant $C_p$, which can be bounded by $\text{diam}(D)/\pi = \sqrt{2}/\pi$ \cite{Payne1960}. Additionally, note that $\lVert y_D\rVert_{L^2(D)}^2=\frac{5}{4}$, $\lVert e_D\rVert_{L^2(D)}^2=(1+4\pi^4)$ and $\lVert u \rVert_{L^2(D)} \leq 1$ for all $u \in \U$.


Results for the first 100 iterations of a single trajectory are displayed in Figure~\ref{fig:objfuncstationaryheat-convex}. Convergence is observed as before, taking the same finer mesh and number of iterations for the reference solution. To approximate objective function values, we use $m=1000$ newly generated samples for each iteration to approximate $j(\tilde{u}_1^n) \approx \hat{j}(\tilde{u}_1^n) = \frac{1}{2m} \sum_{k=1}^m \lVert \hat{y}_{n,k} - y_D \rVert_{L^2(D)}^2 + \frac{\lambda}{2} \lVert \tilde{u}_1^n \rVert_{L^2(D)}$, where $\hat{y}_{k,n} = T_{\omega_{n,k}}(\tu_1^n)$, $\omega_{n,k}$ is the $k^\text{th}$ sample at iteration $n$, and $T_{\omega_{n,k}}(\tu_1^n)$ is the solution of the random PDE with $u = \tu_1^n$ and $a(\omega) = a(\omega_{n,k})$. Note that these additional samples are only used for the generation of a convergence plot; the iterates $u_n$ are still computed using a single realization $\omega_n$.  The error in the objective function value $\hat{j}(\tilde{u}_1^n) - \hat{j}(\tilde{u})$ is $\mathcal{O}(n^{-0.55}),$ which is consistent with the expected theoretical behavior \eqref{eq:efficiency-convexcase}.  The error of iterates for this experiment also displays convergence (for which we do not have a theoretical bound). We note that repeated experiments show that the method with averaging produces less smooth convergence behavior in the objective function. Part of this is related to the fact that step sizes are chosen to be larger; additionally, the regularization constant $\lambda$ also contributes to more smooth convergence behavior in the strongly convex case.  

\begin{figure}
\centering
 \begin{subfigure}{.45\textwidth}
   \includegraphics[width=\linewidth]{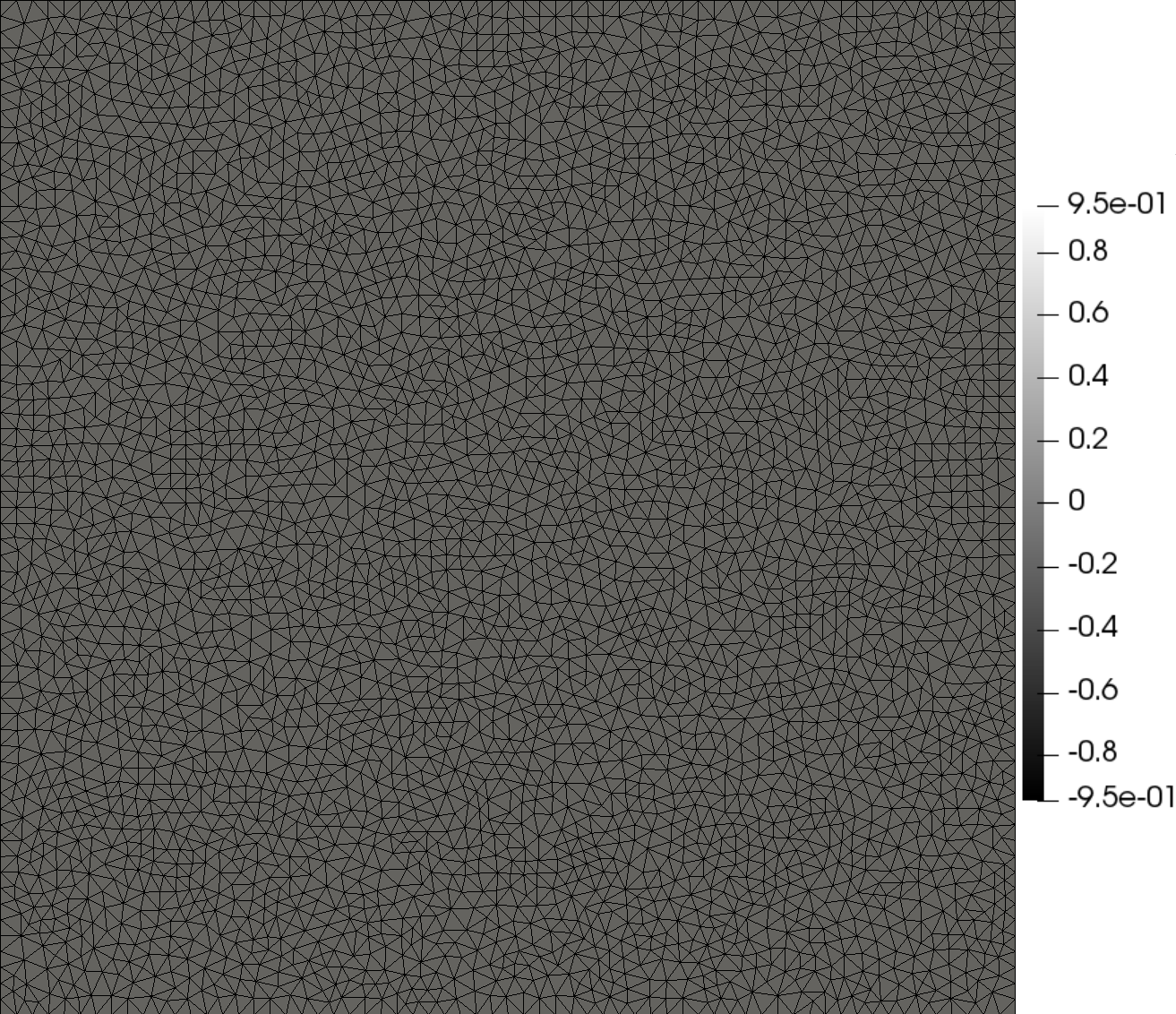}
  \caption{$u_{1}= 0.0$}
 \end{subfigure}
   \begin{subfigure}{.45\textwidth}
   \includegraphics[width=\linewidth]{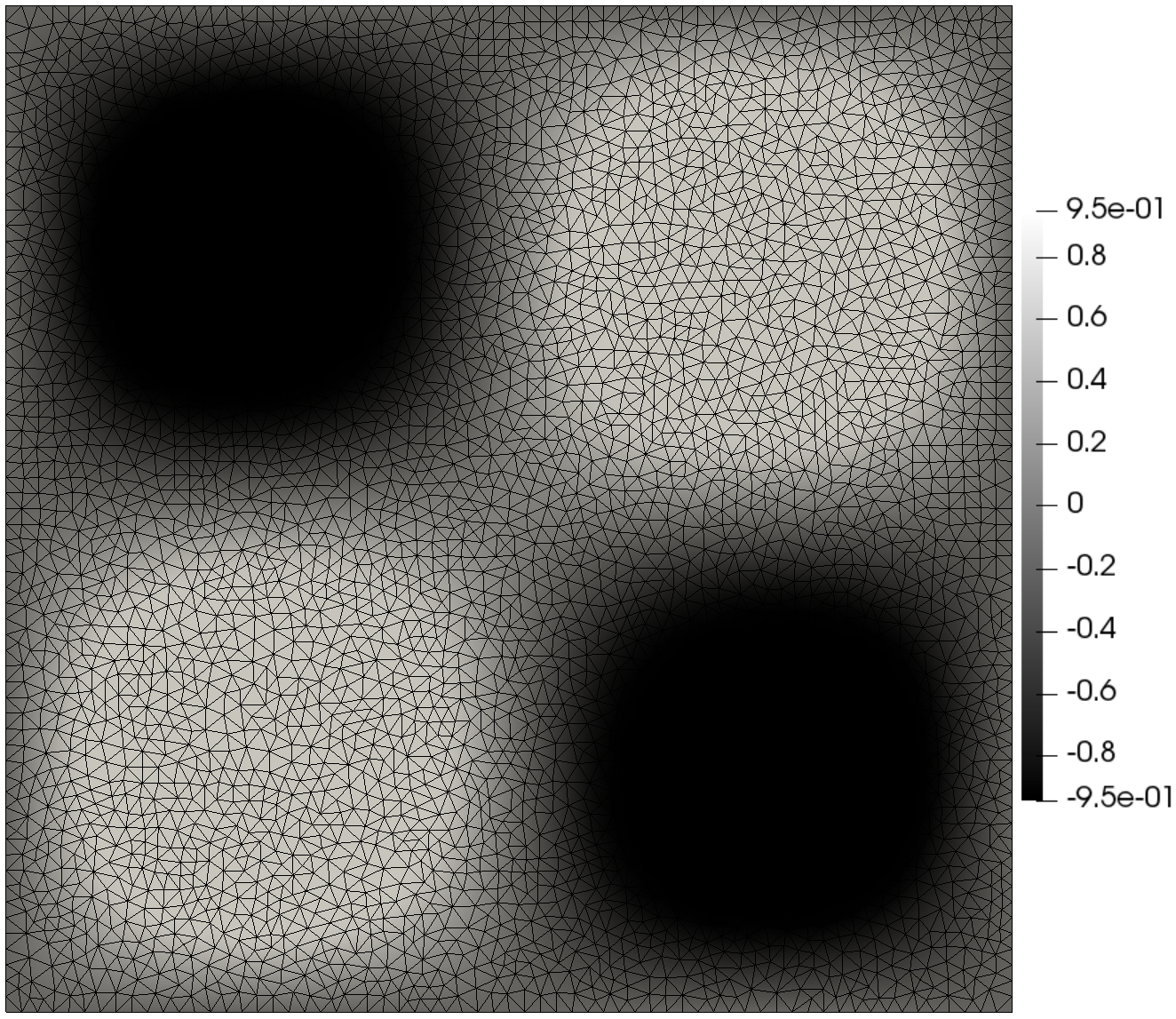}
  \caption{$u_{N}$}
 \end{subfigure}

 \begin{subfigure}{.45\textwidth}
   \includegraphics[width=\linewidth]{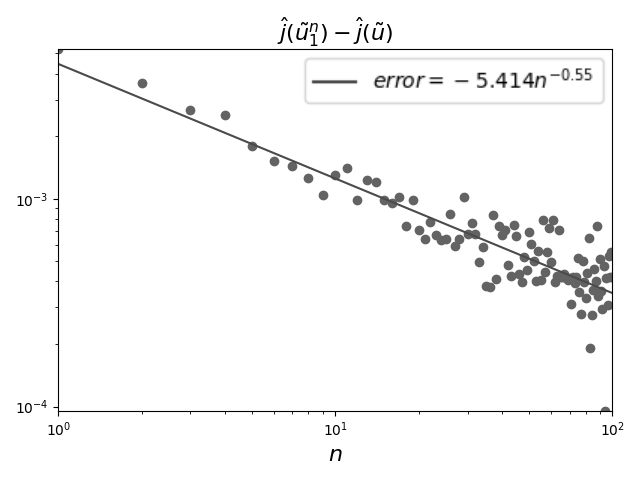}
  \caption{Errors in objective function value}
 \end{subfigure}
   \begin{subfigure}{.45\textwidth}
   \includegraphics[width=\linewidth]{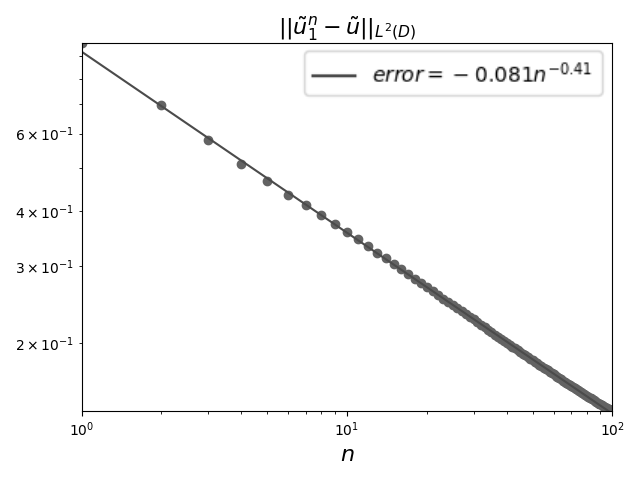}
  \caption{Errors in averaged control}
 \end{subfigure}
   \caption{Convex experiment with $N=100$ iterations.}
\label{fig:objfuncstationaryheat-convex}
 \end{figure}

\section{Conclusion}\label{sec:Conclusion}
In this paper, we presented convergence results for the projected stochastic gradient algorithm for convex problems in Hilbert spaces. The work was motivated by applications in PDE constrained optimal control problems, where a model may contain uncertain parameters. Gradient-based methods are standard tools in deterministic PDE constrained optimization, but the stochastic gradient-based methods from stochastic approximation had been, to the authors' knowledge, up until now undeveloped for problems involving uncertainty. Efficiency estimates and step size rules were derived, which have implications for practical application. Finally, the algorithm was demonstrated on a model problem with random elliptic PDE constraints. Convergence behavior was compared for the strongly convex case and the convex case. In future work, we will investigate the role of the numerical error made by discretization.

\section*{Acknowledgments}
The authors would like to thank the anonymous reviewers for their careful reading of this manuscript.

\appendix \section{Additional Proofs}
\begin{lemma}\label{lemma:recursion-statement}
For a recursion of the form
\begin{equation}
\label{eq:recursion-relation}
 e_{n+1} \leq e_n \left( 1 - \frac{c_1}{n+\nu} +\frac{c_2}{(n+\nu)^2} \right)+\frac{c_3}{(n+\nu)^2},
 \end{equation}
where $e_1, c_2, c_3 \geq 0$ and $c_1 > 1$, it holds that 
\begin{equation}\label{eq:recursion-formula}
e_n \leq \frac{K}{n+\nu},
\end{equation}
with $$K:=\frac{c_3 + e_1 c_2}{c_1-1}, \quad \nu:= \frac{c_3+e_1c_2}{e_1(c_1 -1)}-1.$$
\end{lemma}
\begin{proof}
We show \eqref{eq:recursion-formula} by induction. The statement for $n=1$ is clearly satisfied. For $n>1$, we assume that \eqref{eq:recursion-formula} holds and by \eqref{eq:recursion-relation} we get with $\hat{n}:=n+\nu$
\begin{align*}
 e_{n+1} &\leq  \left(1 - \frac{c_1}{\hat{n}} +\frac{c_2}{\hat{n}^2}\right) \frac{K}{\hat{n}}+\frac{c_3}{\hat{n}^2}\\
 & = \left( \frac{\hat{n}^2-\hat{n}}{\hat{n}^3}\right)K+ \left(\frac{\hat{n}(1-c_1)+c_2}{\hat{n}^3} \right)K+\frac{c_3}{\hat{n}^2}\\
 & \leq \frac{K}{\hat{n}+1}.
\end{align*}
In the last inequality, we used the fact that $\hat{n}^3 \geq \hat{n}(\hat{n}-1)(\hat{n}+1)$ and $[(n-\nu)(1-c_1)+c_2]K + (n+\nu)c_3 \leq 0$ for all $\hat{n}.$
\end{proof}

In the following, $C_p$ denotes the Poincar\'e constant for $D$. The Lax-Milgram Lemma and Poincar\'e inequality can be found in \cite{Evans1998}.

\begin{proof}[Proof of Lemma~\ref{lemma:poissonrandom}]
Let $\omega \in \Omega$ be fixed but arbitrary. The bilinear form $b_\omega: H_0^1(D)\times H_0^1(D) \rightarrow \R$ with $b_\omega(y,v) = \int_D a(x,\omega) \nabla y(x) \cdot \nabla v(x) \D x$
is bounded and coercive, since due to Assumption \eqref{eq:boundsona} and the Poincar\'{e} inequality,
\begin{align*}
|b_\omega(y,v)| &\leq a_{\max} \lVert y \rVert_{H_0^1(D)} \lVert v \rVert_{H_0^1(D)},\\
b_\omega(y,y) &\geq a_{\min} |y|_{H_0^1(D)}^2 \geq \frac{a_{\min}}{C_p^2 +1} \lVert y \rVert_{H_0^1(D)}^2.
\end{align*}
The linear form $l:H_0^1(D) \rightarrow \R$ with $l(v) = \int_D u(x) v(x) \D x $
is bounded since $l(v) \leq \lVert u \rVert_{L^2(D)} \lVert v \rVert_{L^2(D)}$, so by the Lax-Milgram Lemma, there exists a unique solution $y=y(\cdot, \omega)$ satisfying \eqref{eq:Poissonrandom}.
Again using the Poincar\'{e} inequality,
\begin{align*}
\lVert y(\cdot, \omega) \rVert_{L^2(D)}^2 \leq C_p^2 |y(\cdot, \omega)|_{H_0^1(D)}^2 &\leq \frac{C_p^2}{a_{\min}} b_\omega(y,y) \leq {\frac{C_p^2}{a_{\min}}} \lVert u \rVert_{L^2(D)} \lVert y(\cdot,\omega) \rVert_{L^2(D)}.
\end{align*}
The constant is given by $C_1:=\frac{C_p^2}{a_{\min}}.$
\end{proof}

\paragraph{Calculation of the Stochastic Gradient}
We will calculate $\nabla_u J(u,\omega)$ for a fixed realization $\omega \in \Omega$ under the assumption that $a(\cdot, \omega)$ satisfies \eqref{eq:boundsona}. We use the averaged adjoint method from \cite{Sturm2015}, Section 3, as opposed to a formal Lagrangian approach as in \cite{Troeltzsch2009}. This method verifies the existence of the adjoint function $p$ without using the differentiability of the control-to-state operator. Define the Lagrangian $L_\omega:[-\tau,\tau] \times H_0^1(D) \times H_0^1(D) \rightarrow \R$ for the perturbed control $u+t\tu$:
\begin{align*}
L_\omega(t,y,p) &:= \frac{1}{2}\int_D  (y(x,\omega)-y_D(x))^2 \D x + \frac{\lambda}{2} \int_D (u(x)+t\tu(x))^2 \D x \\
& \quad + \int_D a(x,\omega) \nabla y(x,\omega) \cdot \nabla p(x) \D x - \int_D (u(x)+t\tu(x)) p(x) \D x.
\end{align*}
A function $y^t=y^t(\cdot, \omega)$ satisfying $d_p L_\omega(t,y,0)[v] = 0$ for all $v \in H_0^1(D)$ must solve the equation
\begin{equation}\label{eq:perturbedstateequation}
\int_D a(x,\omega) \nabla y^t(x,\omega) \cdot \nabla v(x) \D x - \int_D (u(x)+t\tu(x)) v(x) \D x= 0 \quad \forall v \in H_0^1(D).\end{equation}
Using the same arguments as required for Lemma \ref{lemma:poissonrandom}, for each $t \in [-\tau, \tau]$, $y^t$ is unique. 
Thus, for all $t \in [-\tau, \tau]$, the set
$$P(t,y^t,y^0) := \{p \in H_0^1(D) \, \vert \, \int_0^1 d_y L_\omega(t,sy^t +(1-s)y^0,p)[v] \D s = 0 \,\, \forall v \in H_0^1(D) \}$$
is well-defined. A function $p^t = p^t(\cdot, \omega) \in P(t,y^t,y^0)$ must solve the \textit{averaged adjoint equation} for all $v \in H_0^1(D)$ 
\begin{equation}\label{eq:averagedadjoint}
 \int_D a(x,\omega) \nabla v(x) \cdot \nabla p^t(x,\omega) \D x =\int_D \left(y_D(x) - \frac{y^t(x,\omega) + y^0(x,\omega)}{2} \right) v(x) \D x.
 \end{equation}
By the Lax-Milgram Lemma, $p^t$ is unique for each $t\in[-\tau,\tau]$. The following lemma shows $p^t \rightarrow p^0$ in $H_0^1(D).$

\begin{lemma}\label{lem:strongconvergenceAvgAdjoint}
There exist $\alpha>0$, $\beta>0$ such that for all $t\in [-\tau,\tau]$,
\begin{align}
& \lVert p^t(\cdot,\omega) - p^0(\cdot,\omega) \rVert_{H_0^1(D)} \leq \alpha \lVert y^t(\cdot,\omega) - y^0(\cdot,\omega) \rVert_{L^2(D)} \leq \beta |t| \lVert \tu \rVert_{L^2(D)} \label{eq:secondinequalityaveragedadjoint}.
\end{align}
\end{lemma}

\begin{proof}
Since $p^t\in P(t,y^t,y^0)$ and $p^0 \in P(0,y^0,y^0)$ satisfy \eqref{eq:averagedadjoint}, for all $v \in H_0^1(D)$ it holds that
\begin{equation}\label{eq:equationfortesting}
\int_D a(x,\omega) \nabla v(x) \cdot \nabla (p^t(x,\omega) - p^0(x,\omega)) \D x = \frac{1}{2}\int_D (y^0(x,\omega)-y^t(x,\omega))v(x) \D x.
\end{equation}
We get by testing \eqref{eq:equationfortesting} with $v=p^t-p^0$
\begin{align*}
|p^t(\cdot,\omega) -p^0(\cdot,\omega)|_{H_0^1(D)}^2 & \leq \frac{1}{2a_{\min}} \lVert y^t(\cdot,\omega)-y^0(\cdot,\omega) \rVert_{L^2(D)} \lVert p^t(\cdot,\omega)-p^0(\cdot,\omega) \rVert_{L^2(D)}\\
& \leq {\frac{C_p}{2a_{\min}}} \lVert y^t(\cdot,\omega)-y^0(\cdot,\omega) \rVert_{L^2(D)} | p^t(\cdot,\omega)-p^0(\cdot,\omega)|_{H_0^1(D)}.
\end{align*}
The previous expression yields, using the equivalence of the $H^1$ norm and seminorm, the first inequality in \eqref{eq:secondinequalityaveragedadjoint}. For the second inequality, $y^t$ and $y^0$ must satisfy \eqref{eq:perturbedstateequation}, so we have
\begin{equation}\label{eq:equationfortesting-state}
 \int_D a(x,\omega) \nabla (y^t(x,\omega)-y^0(x,\omega)) \cdot \nabla v(x) \D x= \int_D t \tu(x) v(x) \D x \quad \forall v \in H_0^1(D).
\end{equation}
Testing \eqref{eq:equationfortesting-state} with $v=y^t-y^0$ yields the bound
$$\lVert y^t(\cdot,\omega) - y^0(\cdot,\omega) \rVert_{L^2(D)} \leq \frac{C_p^2}{a_{\min}} |t| \lVert \tu \rVert_{L^2(D)}.$$
\end{proof}

\begin{proof}[Proof of Proposition~\ref{prop:gradient}]
By the mean value theorem, the averaged adjoint equation \eqref{eq:averagedadjoint} implies for for $p^t \in P(t,y^t,y^0)$ that
\begin{equation}\label{eq:LagrangianIdentity}
L_\omega(t,y^t, p^t)=L_\omega(t,y^0,p^t),
\end{equation} since $\int_0^1 d_y L_\omega(t,sy_t+(1-s)y_0,p)[y^t-y^0] \D s = 0$ by definition of $P(t, y^t, y^0)$.
Since  $y^0$ satisfies \eqref{eq:perturbedstateequation} for $t=0$, we also have $L_\omega(0,y^0,p^0)=L_\omega(0,y^0,p^t)$. Thus,
\begin{equation}\label{eq:diffquotient}
\begin{aligned}
&J(u+t\tu,\omega)-J(u, \omega) = L_\omega(t,y^t,p^t) - L_\omega(0,y^0,p^0)  = L_\omega(t,y^0,p^t) -L_\omega(0,y^0,p^t) \\
& \quad =  \frac{\lambda}{2} \int_D (u(x)+t \tu(x))^2 \D x - \int_D (u(x)+t\tu(x)) p^t(x,\omega) \D x \\
& \quad \quad - \frac{\lambda}{2} \int_D u(x)^2 \D x+ \int_D u(x) p^t(x,\omega)\D x\\
& \quad = t \lambda \int_D u(x) \tu(x) \D x + t^2 \frac{\lambda}{2}\int_D \tu(x)^2 \D x -  t \int_D \tu(x) p^t(x,\omega) \D x.
\end{aligned}
\end{equation}
Dividing \eqref{eq:diffquotient} by $t \neq 0$ and passing to the limit, and using the fact that $p^t \rightarrow p^0$ in $H^1(D)$ by Lemma~\ref{lem:strongconvergenceAvgAdjoint},
$$d_u J (u,\omega)[\tu] = \lim_{t \rightarrow 0} \frac{J(u+t\tu,\omega)-J(u, \omega)}{t} =  \int_D (\lambda u(x) - p^0(x,\omega)) \tu(x) \D x,$$
where $p^0$ solves the problem
\begin{equation*}
 \int_D \left(y^0(x,\omega) - y_D(x) \right) v(x) \D x+ \int_D a(x,\omega) \nabla v(x) \cdot \nabla p(x,\omega) \D x = 0 \quad \forall v \in H_0^1(D),
 \end{equation*}
which is the same as \eqref{eq:Poissonrandom-adjoint} with $y=y^0$. The $L^2$-stochastic gradient $\nabla_u J (u,\omega)=\lambda u - p^0$ is the Riesz representation of $d_u J (u,\omega)[\tu].$
\end{proof}

\begin{proof}[Proof of Lemma~\ref{lemma:poissonrandom-adjoint}]
With the same arguments as in the proof for Lemma \ref{lemma:poissonrandom}, the existence and uniqueness of a solution $p(\cdot,\omega)$ to \eqref{eq:Poissonrandom-adjoint} can be established using the Lax-Milgram Lemma. Then inequality \eqref{eq:Aprioribounds_Poisson_stochastic_adjoint} follows from
\begin{align*}
\lVert p(\cdot, \omega) \rVert_{L^2(D)}^2 \leq C_p^2 |p(\cdot, \omega)|_{H_0^1(D)}^2 \leq {\frac{C_p^2}{a_{\min}}} \lVert y_D - y(\cdot,\omega) \rVert_{L^2(D)} \lVert p(\cdot,\omega) \rVert_{L^2(D)}
\end{align*}
with $C_2:=\frac{C_p^2}{a_{\min}}.$
\end{proof}

\bibliographystyle{siamplain}
\bibliography{references}

\begin{thebibliography}{10}

\bibitem{Alber1998}
{\sc Y.~Alber, A.~Iusem, and M.~Solodov}, {\em On the projected subgradient
  method for nonsmooth convex optimization in a {H}ilbert space}, Math.
  Program., 81 (1998), pp.~23--35.

\bibitem{Alnes2015}
{\sc M.~S. Aln{\ae}s, J.~Blechta, J.~Hake, A.~Johansson, B.~Kehlet, A.~Logg,
  C.~Richardson, J.~Ring, M.~E. Rognes, and G.~N. Wells}, {\em The {FEniCS}
  project version 1.5}, Arch. Numer. Softw., 3 (2015).

\bibitem{Babuska2004}
{\sc I.~Babuska, R.~Tempone, and G.~E. Zouraris}, {\em Galerkin finite element
  approximations of stochastic elliptic partial differential equations}, SIAM
  J. Numer. Anal., 42 (2004), pp.~800--825.

\bibitem{Barty2005}
{\sc K.~Barty, J.-S. Roy, and C.~Strugarek}, {\em A perturbed gradient
  algorithm in {H}ilbert spaces}, Optim. Online,  (2005).

\bibitem{Bertsekas2000}
{\sc D.~Bertsekas and J.~Tsitsiklis}, {\em Gradient convergence in gradient
  methods with errors}, {SIAM} J. Optim., 10 (2000), pp.~627--642.

\bibitem{Chen2002}
{\sc X.~Chen and H.~White}, {\em Asymptotic properties of some projection-based
  {R}obbins-{M}onro procedures in a {H}ilbert space}, Stud. Nonlinear Dyn.
  Econom., 6 (2002), pp.~1--53.

\bibitem{Cruz2016}
{\sc J.~Y.~B. Cruz and W.~de~Oliveira}, {\em On weak and strong convergence of
  the projected gradient method for convex optimization in {H}ilbert spaces},
  Numer. Funct. Anal.Optim., 37 (2016).

\bibitem{Culioli1990}
{\sc J.-C. Culioli and G.~Cohen}, {\em Decomposition/coordination algorithms in
  stochastic optimization}, SIAM J. Control Optim., 28 (1990), pp.~1372--1403.

\bibitem{Evans1998}
{\sc L.~Evans}, {\em Partial Differential Equations}, vol.~Graduate Studies in
  Mathematics vol. 19, American Mathematical Society, Providence, R.I., 1998.

\bibitem{George2006}
{\sc A.~George and W.~Powell}, {\em Adaptive stepsizes for recursive estimation
  with applications in approximate dynamic programming}, Mach. Learn., 65
  (2006), pp.~167--198.

\bibitem{Goldstein1988}
{\sc L.~Goldstein}, {\em Minimizing noisy functionals in {H}ilbert space: An
  extension of the {K}iefer-{W}olfowitz procedure}, J. Theoret. Probab., 1
  (1988).

\bibitem{Hinze2009}
{\sc M.~Hinze, R.~Pinnau, M.~Ulbrich, and S.~Ulbrich}, {\em Optimization with
  PDE Constraints}, Springer, 2009.

\bibitem{Hou2011}
{\sc L.~Hou, J.~Lee, and H.~Manouzi}, {\em Finite element approximations of
  stochastic optimal control problems constrained by stochastic elliptic
  {PDE}s}, J. Math. Anal. Appl., 384 (2011), pp.~87--103.

\bibitem{Kiefer1952}
{\sc J.~Kiefer and J.~Wolfowitz}, {\em Stochastic estimation of the maximum of
  a regression function}, Ann. Math. Stat., 23 (1952), pp.~462--466.

\bibitem{Kouri2014}
{\sc D.~Kouri, M.~Heinkenschloss, D.~Ridzal, and B.~V.~B. Waanders}, {\em
  Inexact objective function evaluations in a trust-region algorithm for
  {PDE}-constrained optimization under uncertainty}, SIAM J. Sci. Comput., 36
  (2014).

\bibitem{Kushner2003}
{\sc H.~Kushner and G.~Yin}, {\em Stochastic Approximation and Recursive
  Algorithms and Applications}, Springer-Verlag New York, 2003.

\bibitem{Lord2014}
{\sc G.~Lord, C.~Powell, and T.~Shardlow}, {\em An Introduction to
  Computational Stochastic PDEs}, Cambridge University Press, 2014.

\bibitem{Nemirovski2009}
{\sc A.~Nemirovski, A.~Juditsky, G.~Lan, and A.~Shapiro}, {\em Robust
  stochastic approximation approach to stochastic programming}, SIAM J. Optim.,
  19 (2009), pp.~1574--1609.

\bibitem{Nixdorf1984}
{\sc R.~Nixdorf}, {\em An invariance principle for a finite dimensional
  stochastic approximation method in a {H}ilbert space}, J. Multivariate Anal.,
  15 (1984), pp.~252--260.

\bibitem{Payne1960}
{\sc L.~Payne and H.~Weinberger}, {\em An optimal {P}oincar\'e inequality for
  convex domains}, Arch. Rational Mech. Anal., 5 (1960), pp.~286--292.

\bibitem{Pflug1996}
{\sc G.~C. Pflug}, {\em Optimization of Stochastic Models: The Interface
  Between Simulation and Optimization}, Springer, 1996.

\bibitem{Poljak1967}
{\sc B.~Polyak}, {\em A general method of solving extremum problems}, Soviet
  Mathematics Doklady, 8 (1967).

\bibitem{Polyak1992}
{\sc B.~Polyak and A.~Juditsky}, {\em Acceleration of stochastic approximation
  by averaging}, {SIAM} J. Control Optim., 30 (1992), pp.~838--855.

\bibitem{Robbins1951}
{\sc H.~Robbins and S.~Monro}, {\em A stochastic approximation method}, Ann.
  Math. Statist., 22 (1951), pp.~400--407.

\bibitem{Rosseel2012}
{\sc E.~Rosseel and G.~Wells}, {\em Optimal control with stochastic {PDE}
  constraints and uncertain controls}, Comput. Methods Appl. Mech. Engrg.,
  (2012), pp.~152--167.

\bibitem{Shapiro2003}
{\sc A.~Shapiro}, {\em Handbook in Operations Research and Management Science},
  vol.~10, Elsevier, 2003, ch.~Monte Carlo Sampling Methods, pp.~353--425.

\bibitem{Sturm2015}
{\sc K.~Sturm}, {\em Minimax {L}agrangian approach to the differentiability of
  nonlinear {PDE} constrained shape functions without saddle point assumption},
  SIAM J. Control Optim., 53 (2015), pp.~2017--2039.

\bibitem{Troeltzsch2009}
{\sc F.~Tr\"oltzsch}, {\em Optimale Steuerung partieller
  Differentialgleichungen}, Vieweg + Teubner, 2nd~ed., 2009.

\bibitem{Venter1966}
{\sc J.~Venter}, {\em On {D}voretzky stochastic approximation theorems}, Ann.
  Math. Stat., 37 (1966), pp.~1534--1544.

\bibitem{Yin1990}
{\sc G.~Yin and Y.~Zhu}, {\em On {$H$}-valued {R}obbins-{M}onro processes}, J.
  Multivariate Anal., 34 (1990), pp.~116--140.

\end{thebibliography}
\end{document}